\documentclass[12pt]{article}

\usepackage[utf8]{inputenc}   
\usepackage[T1]{fontenc}
\usepackage[a4paper]{geometry}

\pdfoutput=1 

\usepackage{color}
\usepackage{dsfont}
\usepackage{graphicx}
\usepackage{amsmath, amssymb,mathrsfs}
\usepackage{amsthm}
\usepackage{caption}
\usepackage{multirow}
\usepackage{array}

\newcommand{\ep}{\varepsilon}
\newcommand{\R}{\mathbb{R}}
\newcommand{\N}{\mathbb{N}}
\newcommand{\eps}{\varepsilon}
\newcommand{\T}{\mathbb{T}}

\newtheorem{theorem}{Theorem}[section]

\newtheorem{lemma}[theorem]{Lemma}

\newtheorem{proposition}[theorem]{Proposition}
\newtheorem{remark}[theorem]{Remark}

\RequirePackage{fix-cm}

\usepackage{graphicx}

\usepackage{latexsym}
\usepackage{amsmath, amssymb,mathrsfs}
\begin{document}

\title{Analysis of a time-dependent problem of mixed migration and population dynamics}

\author{
Francois Castella\footnote{Irmar, Universit\'e de Rennes 1, Campus de Beaulieu, 35042 Rennes,
francois.castella\@@univ-rennes1.fr},
Philippe Chartier\footnote{INRIA, Universit\'e de Rennes 1, Campus de Beaulieu, 35042 Rennes,
philippe.chartier\@@inria.fr}, Julie Sauzeau\footnote{Irmar, Universit\'e de Rennes 1, Campus de Beaulieu, 35042 Rennes,
julie.sauzeau\@@univ-rennes1.fr}}

\maketitle

\begin{abstract}
In this work, we consider a system of differential equations modeling the dynamics of some populations of preys and predators, moving in space according to rapidly oscillating time-dependent transport terms, and interacting with each other through a Lotka-Volterra term. These two contributions naturally induce two separated time-scales in the problem. A generalized center manifold theorem is derived to handle the situation where the linear terms are depending on the fast time in a periodic way. The resulting equations are then amenable to averaging methods. 
As a product of these combined techniques, one obtains an autonomous differential system in reduced dimension whose dynamics can be analyzed in a much simpler way as compared to original equations. Strikingly enough, this system is of Lotka-Volterra form with modified coefficients. Besides, a higher order perturbation analysis allows to show that the oscillations on the original model destabilize the cycles of the averaged Volterra system in a way that can be explicitely computed.\\ \\
\end{abstract}

\noindent
Keywords: Center manifold, Shadowing, Averaging, Lotka-Volterra, Perron-Frobenius, population dynamics.

Mathematics Subject Classification (2010):
35B25;  92D40

\section{Introduction}
\label{intro}
This paper is concerned with the analytical study of the dynamics of a prey-predator model and is a follow-up of \cite{Castella}.
 The model under consideration here takes into account both interactions between species {\em and} spatial migrations.
As such it is a strict  elaboration of the well-known Lotka-Volterra equations.  In particular, a fundamental feature of the operating dynamics that we wish to mention right away is the occurrence of two time-scales, accounting for the fact that spatial evolutions are vastly faster than demographical ones. 

In this still simplified version, the space is discretized into $N$ distinct sites amongst which  species move rapidly (i.e. change from one site to another within, say, a few hours). These migrations are described by two linear operators corresponding to preys, on the one hand, and to predators, on the other hand, which both depend periodically on time in a highly-oscillatory way.  One may think of preys and predators migrating on the time scale of an hour, say, with migrations rates which vary on the same time scale, due to dayly variations of the environment. This is typically the case for plancton, whose motion in the vectorial direction depends on the light brought by the sun during the day. In addition, these operators are assumed to preserve the number of individuals, so that migration and demographic terms remain independent in the equations (individuals can not die while migrating). 

As for predator-prey interactions, they are described by a term of Lotka-Volterra type which may differ from one site to another, that is to say, spatial characteristics may vary (for instance owing to more abundant food or more spots to hide). These interactions typically become apparent over intervals of time that can be gauged in months. For this reason, we shall introduce the small parameter $\eps$ defined as the ratio between the two present time-scales (migrations over predator-prey interactions).  

The complete model shall be  presented with full details in Section 2. However, in this introductory section, it is enlightening to describe it in an abstract and concise form as follows
\begin{eqnarray}
{\frac{d}{d t}(\mathrm{populations})(t)}&=\frac{1}{\eps}(\mathrm{migration ~ term})\left(\frac{t}{ \eps}\right) 
\cdot (\mathrm{populations})(t) \nonumber\\
{}&{\quad\quad\quad+(\mathrm{prey-predator ~ interactions})}\nonumber
\label{int}
\end{eqnarray}
where the migration term is periodic in $t/\eps$. The aim of this work is to conduct an analysis of the dynamics in the limit $\eps \rightarrow 0$ and to draw conclusions from the resulting asymptotic model: in order to do so, the original equations are reduced through several changes of variables to a form which is amenable to center manifold techniques. The center manifold which is then constructed in Section \ref{sect:manifold}
has the peculiarity to depend periodically on the fast-time variable. In the same section, we will then prove that it can be approximated up to every order in the small parameter $\eps$ by the appropriately truncated solution  of a partial differential equation: it is noteworthy to mention that this solution can be computed explicitly through a recursive relation. 

As a result of the center manifold theorem of Section \ref{sect:manifold}, one obtains a highly-oscillatory time-dependent differential system (the fast oscillations originating from $\theta=t/\eps$ in $h(\cdot,\theta)$). In order to grasp the essential dynamics, we thus derive the corresponding averaged equations in  which  variable $\theta=t/\eps$ has been integrated. In Section \ref{sect:averaging} the procedure is thus described up to arbitrary order errors in $\eps$ at the additional burden of a periodic change of variables. Explicit equations up to order $1$ in $\eps$ are presented. 

In the last section, the full methodology is worked through for an example implying two sites. Interestingly enough, the limit equations are still of Lotka-Volterra type, and the coefficients are an average in space and time of the original Lotka-Volterra coefficients and on those of the transfer operators.


\section{Description of the model}
\label{sec:1}
In this first paper, we content ourselves with a finite dimensional description of our problem. Accordingly, we  consider a  discretization of the spatial domain  into $N\in \N^*$ subdomains,  pick up a small dimensionless parameter $\varepsilon$ expressing the ratio between the time-scales of migrations and demographic evolution and denote by $p_i^\ep(t),~i\in[\![ 1,N]\!]$ the number of preys occupying the $i^{\mbox{th}}$ site  at time $t$ and by 
$q_i^\ep(t),~i\in[\![ 1,N]\!]$ the corresponding number of predators. Introducing the  vectors
\begin{eqnarray}
{p}^\ep(t)=\left( p^\ep_1(t) , \cdots,   p^\varepsilon_N(t)\right)^T \quad \mbox{ and } \quad
{q}^\ep(t)=\left(  q^\ep_1(t),  \cdots  ,  q^\varepsilon_N(t) \right)^T,\nonumber
\end{eqnarray}
the initial-value problem can be written as
\begin{equation}\left\{\begin{array}{lll}
\frac{\mathrm{d}p^\varepsilon(t)}{\mathrm{dt}} &=\frac{1}{\ep} K_p\left(\frac{t}{\ep}\right) {p}^\varepsilon(t)+f\left({p}^\varepsilon(t),{q}^\varepsilon(t)\right), & p^\ep(0)=p_0
\vspace{2mm}\\
\frac{\mathrm{d}q^\varepsilon(t)}{\mathrm{dt}} &=\frac{1}{\ep} K_q\left(\frac{t}{\ep}\right) {q}^\varepsilon(t)+g\left({p}^\varepsilon(t),{q}^\varepsilon(t)\right), & q^\ep(0)=q_0
\label{1}
\end{array}\right.\end{equation}
where $K_p$ and $K_q$ are time-dependent transport matrices defined by
\begin{eqnarray} (K_p(t))_{i,j}=\sigma^p_{i,j}(t) \mbox{ for } i\neq j, \quad 
(K_p(t))_{i,i}=-\sum\limits_{k=1}^N \sigma^p_{k,i}(t) \nonumber
\end{eqnarray} 
and by similar equations for $K_q$. 
Here, the rates of transfer $\sigma^p_{i,j}(t)$ and $\sigma^q_{i,j}(t)$ are the proportions of preys, respectively predators,  moving from site $j$ to site $i$ at time $t$. These functions are assumed to be positive and periodic with period $T=2\pi$. Note that, as a direct consequence of these definitions, $\mathbf{1}=(1,\ldots,1)^T \in \R^N$ is a left-eigenvector of both $K_p$ and $K_q$, i.e.  $\mathbf{1}^T K_p(t)=\mathbf{1}^T K_q(t)=0$.

As for the functions $f$ and $g$, they model non-linear interactions of Lotka-Volterra type between species: for all integers $1 \leq i \leq N$
\begin{eqnarray}
f_i\left({p}^\ep(t),{q}^\ep(t)\right) &= a_{p,i} \, p_i^\varepsilon(t)-b_{p,i} \,p_i^\varepsilon(t) \,q_i^\varepsilon(t)   \nonumber\\
g_i\left(p^\ep(t),{q}^\ep(t)\right)  &= -a_{q,i} \,q_i^\varepsilon(t)+b_{q,i} \,p_i^\varepsilon(t) \, q_i^\varepsilon(t).\nonumber
\end{eqnarray}
Coefficient $a_{p,i}$, assumed to be independent of time, is the birth-rate of preys on site $i$, while $a_{q,i}$ is the death-rate of predators on the site $i$ . In the same way, $b_{p,i}$ is the death-rate of preys on site $i$ caused by the predators, while $b_{q,i}$ is the birth-rate of predators due to the presence of preys. All those quantities are non-negative.

\begin{remark} 
For the sake of simplicity, interactions between species are modeled here  with constant coefficients for the species interactions. However, the following theorems would remain true with time-dependent coefficients,
 $a_{p,i}=a_{p,i}(t)$, and so on. 
\end{remark}

\section{Analysis and reduction of the system}
\label{sec:2}
\subsection{Main properties of the linear part of the system}
\label{sec:2.1}
In this subsection, we present the spectral properties of our transport operators $K_p(t)$ and $K_q(t)$. They determine the modifications which are necessary to bring our system into a form amenable to center manifold techniques.

\begin{lemma} \label{SpecK}
For all $\theta\in\T$, $0 \in \mbox{Sp}(K_p(\theta))$ and is simple, while other eigenvalues have a strictly negative real part. The right-eigenvector $p_{eq}(\theta)$ associated to $0$ can be chosen as a smooth function w.r.t. $\theta$  satisfying $p_{eq}\cdot\mathbf{1}=1$. Moreover the following property holds true:
\begin{eqnarray}
\exists \beta >0, \quad \forall \theta \in \T, \quad \forall \lambda \in \mbox{Sp}(K_p(\theta))/\{0\}, \quad \Re{(\lambda)}\leq-2\beta.\nonumber
\end{eqnarray}
Besides,  denoting $\mathcal{E}_0=\{z\in\R^N,z\cdot\mathbf{1}=0\}$, one has 
$$
\forall \theta \in \T,\quad \mathcal{E}_0 \oplus \mathrm{Span}(p_{eq}(\theta)) = \R^N.
$$
The same properties hold for matrix $K_q$, with the right-eigenvector $q_{eq}$ associated to  eigenvalue $0$ and the supplementary space is again $\mathcal{E}_0$.
\end{lemma}

\begin{proof} 
It is an easy application of 
the Perron-Frobenius theorem.
\end{proof}

\subsection{Reduction of the system}
\label{sec:2.2}
Starting from differential system $(\ref{1})$, we wish to prove that there exists a function $h(\cdot,\theta)$, periodic in $\theta\in\T$, such that some of the solution-components (collected in a vector denoted $Z \in \R^{2N-2}$) of $(\ref{1})$ can be expressed in terms of other solution-components (collected in a vector denoted $X\in \R^2$) and $\theta\in\T$. The center manifold related to our problem shall then be the set $\{h(X,\theta),~X\in\R^2,~\theta\in\T\}$. Generally speaking, the existence of such a function is stated for equations where the transport-matrices $K_p(\theta)$ and $K_q(\theta)$ have eigenvalues lower than a strictly negative constant $-\beta$ for all $\theta\in\T$, or, alternatively, are just vanishing. In order to adapt the proof of existence of a center manifold for  $(\ref{1})$, it is thus necessary  to recast  system $(\ref{1})$ in the form 
\begin{equation}\left\{\begin{array}{ll}
\frac{\mathrm{d}X}{\mathrm{dt}}&=F\left(X,Z,\frac{t}{\ep}\right)\vspace{1mm}\\ 
\frac{\mathrm{d}Z}{\mathrm{dt}}&=\frac{1}{\varepsilon}B\left(\frac{t}{\ep}\right)Z+G\left(X,Z,\frac{t}{\ep}\right)
\end{array}\right.
\label{VarRef}
\end{equation}
where $B(\theta)$ is a periodic matrix related to an exponentialy decreasing resolvent and where $F$ and $G$ have bounded  derivatives w.r.t.  $X$ and $Z$ up to order $r$, are differentiable and periodic w.r.t. $\theta\in\T$, and, in addition, are globally bounded and Lipschitz on $\R^n \times \R^m \times \T$, with a Lipschitz constant independent of $\ep$. Given that matrix $K_p(\theta)$ has a simple eigenvector $p_{eq}(\theta)$ associated to $0$ and other eigenvalues lower than a constant $-\beta$ (see Lemma $\ref{SpecK}$), our first step will consist in treating separately the equation corresponding to $p_{eq}(\theta)$ and the projected equations on $\mathcal{E}_0=\{z\in\R^N,z\cdot\mathbf{1}=0\}$. 
The equations corresponding to variable $q$ will be treated accordingly. Our second step will consist in removing the remaining stiff term through a time-dependent change of variables. Our last step will consist in localising in $X$ and $Z$ the remaining Lotka-Volterra part of the equation, to introduce the functions $F$ and $G$.\\ \\
{\bf Fisrt step:} 
Introducing $x_p(t) = \mathbf{1} \cdot p^\varepsilon(t)$ and $y_p(t) =p^\eps(t)-x_p(t) p_{eq}\left(\frac{t}{\eps}\right) \in\mathcal{E}_0$, one has 
 \begin{equation}\left\{\begin{array}{ll}
 \dot{x}_p&= f^x_1\left(x_p,x_q,y_p,y_q,\frac{t}{\ep}\right)\vspace{2mm}\\
 \dot{y}_p&=\frac{1}{\ep} K_p\left(\frac{t}{\ep}\right) y_p-\frac{1}{\ep}x_p \dot{p}_{eq}\left(\frac{t}{\ep}\right) + f^y_1\left(x_p,x_q,y_p,y_q,\frac{t}{\ep}\right)\vspace{2mm}\\
 \dot{x}_q&= g^x_1\left(x_p,x_q,y_p,y_q,\frac{t}{\ep}\right)\vspace{2mm}\\
 \dot{y}_q&=\frac{1}{\ep} K_q\left(\frac{t}{\ep}\right) y_q-\frac{1}{\ep}x_q \dot{q}_{eq}\left(\frac{t}{\ep}\right) + g^y_1\left(x_p,x_q,y_p,y_q,\frac{t}{\ep}\right)
\end{array}\right.
\label{Main2}
\end{equation}
 with 
  \begin{eqnarray}
 f^x_1\left(x_p,x_q,y_p,y_q,\theta\right)&=\mathbf{1}\cdot f\left(x_p p_{eq}\left(\theta\right)+y_p,x_q q_{eq}\left(\theta\right)+y_q\right), \nonumber \\
 f^y_1\left(x_p,x_q,y_p,y_q,\theta\right)&=f\left(x_p p_{eq}\left(\theta\right)+y_p,x_q q_{eq}\left(\theta\right)+y_q\right)\nonumber\\ 
 &\qquad- f^x_1\left(x_p,x_q,y_p,y_q,\theta\right) \,p_{eq}\left(\theta\right),\nonumber
 \end{eqnarray}
and similarly for $g^x_1$ and $g^y_1$. Now, let $\tilde{K}_p$ and $\tilde{K}_q$ denote the projections on $\mathcal{E}_0$ of the matrices $K_p$ and $K_q$. More precisely, $\tilde{K}_p$ can be defined as $\tilde{K}_p = J_1 K_p J_2$ (and similarly for $\tilde{K}_q$), where $J_1$ is the $(N-1)\times (N-1)$ identity matrix with an additional column of zeros and $J_2$ is $(N-1)\times (N-1)$ identity matrix with an additional row of $-1$. It is then clear that for $y\in\mathcal{E}_0$, $J_2 J_1 y = y$, so that to each eigenvector $y \in \mathcal{E}_0$ of $K_p$ for the eigenvalue $\lambda$, we can associate an eigenvector $J_1 y$ of $\tilde{K}_p$, as can be seen from the relation
$$
\tilde{K}_p (J_1 y) = J_1 K_p (J_2 J_1) y = J_1 K_p  y = \lambda (J_1 y). 
$$
As a consequence, $\tilde{K}_p$ and $\tilde{K}_q$ are invertible matrices, with eigenvalues of real part respectively smaller than $-\beta_p<0$ and $-\beta_q<0$. Equation 
$$ 
\dot{y}_p=\frac{1}{\ep} K_p\left(\frac{t}{\ep}\right) y_p-\frac{1}{\ep}x_p \dot{p}_{eq}\left(\frac{t}{\ep}\right) + f^y_1\left(x_p,x_q,y_p,y_q,\frac{t}{\ep}\right)
$$
can finally  be projected on $\mathcal{E}_0$ by pre-multiplying by $J_1$:
$$
\dot{\tilde y}_p=\frac{1}{\ep} \tilde{K}_p\left(\frac{t}{\ep}\right) \tilde{y}_p -\frac{1}{\ep}x_p \, J_1 \dot{p}_{eq}\left(\frac{t}{\ep}\right) + J_1 f^y_1\left(x_p,x_q,y_p,y_q,\frac{t}{\ep}\right),
$$
and system (\ref{Main2}) is transformed into 
 \begin{equation}\left\{\begin{array}{ll}
 \dot{x}_p&= \tilde f^x_1\left(x_p,x_q,\tilde y_p,\tilde y_q,\frac{t}{\ep}\right)\vspace{2mm}\\
 \dot{\tilde y}_p&=\frac{1}{\ep} \tilde {K}_p\left(\frac{t}{\ep}\right) \tilde{y}_p-\frac{1}{\ep}x_p \dot{\tilde p}_{eq}\left(\frac{t}{\ep}\right) + \tilde f^y_1\left(x_p,x_q,\tilde y_p,\tilde y_q,\frac{t}{\ep}\right)\vspace{2mm}\\
 \dot{x}_q&= \tilde g^x_1\left(x_p,x_q,\tilde y_p,\tilde y_q,\frac{t}{\ep}\right)\vspace{2mm}\\
 \dot{\tilde y}_q&=\frac{1}{\ep} \tilde{K}_q\left(\frac{t}{\ep}\right) \tilde{y}_q-\frac{1}{\ep}x_q \dot{\tilde q}_{eq}\left(\frac{t}{\ep}\right) + \tilde{g}^y_1\left(x_p,x_q,\tilde y_p,\tilde y_q,\frac{t}{\ep}\right)
\end{array}\right.
\label{Main3}
\end{equation}

{\bf Second step:} In order to get rid of the stiff terms in $\frac{1}{\eps} x_p\dot{\tilde p}_{eq}$ and  $\frac{1}{\eps} x_q\dot{\tilde q}_{eq}$ in system $(\ref{Main2})$, we now introduce a change of variables of the form
$$
z_p(t)=\tilde{y}_p(t)-h_p^0\left(x_p(t),\frac{t}{\ep}\right)
$$
where the function $h_p^0(x,\theta)$ is required to be periodic in $\theta$. The aim is obtain a differential equation of the form 
$$
\dot{z}_p(t) =\frac{1}{\ep}\tilde{K}_p\left(\frac{t}{\ep}\right) z_p(t) + \hat{f}\left(x_p,x_q,z_p,z_q,\frac{t}{\ep}\right) 
$$
where $\hat{f}$ is a function without any pre-factor in $\frac{1}{\eps}$. Differentiating the previous equation w.r.t. time leads to
\begin{eqnarray}
\dot{z}_p(t)&= \dot{\tilde y}_p(t) - \partial_x h_p^0\left(x_p(t),\frac{t}{\ep}\right) \dot{x}_p(t) -\frac{1}{\ep} \partial_\theta h_p^0\left(x_p(t),\frac{t}{\ep}\right) \nonumber\\
&= \frac{1}{\ep}\tilde{K}_p\left(\frac{t}{\ep}\right) z_p(t)- \partial_x h_p^0\left(x_p(t),\frac{t}{\ep}\right)  \tilde{f}^x_1\left(x_p,x_q,\tilde y_p,\tilde y_q,\frac{t}{\ep}\right) \nonumber\\&+ \tilde{f}^y_1\left(x_p,x_q,\tilde y_p,\tilde y_q,\frac{t}{\ep}\right)\nonumber\\
&\qquad\qquad -\frac{1}{\ep} \partial_\theta h_p^0\left(x_p(t),\frac{t}{\ep}\right)  +\frac{1}{\ep}\tilde{K}_p\left(\frac{t}{\ep}\right) h_p^0\left(x_p(t),\frac{t}{\ep}\right)-\frac{1}{\ep}x_p \dot{\tilde p}_{eq} \nonumber
\end{eqnarray}
from which it becomes clear that $h_p^0(x,\theta)$ should be taken as a periodic solution of the following equation 
\begin{equation}
\partial_\theta h_p^0\left(x,\theta\right) =\tilde{K}_p\left(\theta\right) h_p^0\left(x,\theta\right) - x \, \dot{\tilde p}_{eq}(\theta).
\label{Prec}
\end{equation}
In order to solve the previous equation, we then consider $\tilde{R}_p(\theta,s)$ its resolvent, defined as the solution of 
$$
\partial_\theta\tilde{R}_p(\theta,s)=\tilde{K}_p\left(\theta\right) \tilde{R}_p(\theta,s),\quad \tilde{R}_p(s,s)=\mathrm{Id}.
$$
The solution of $(\ref{Prec})$ can be obtained as
$$
h_p^0(x,\theta) = \tilde R_p(\theta,0) h_p^0(\cdot,0) - \int_0^\theta \tilde R_p(\theta,\varphi) x \, \dot{\tilde p}_{eq}(\varphi) d\mathrm{\varphi}
$$
For $h_p^0$ to be periodic with period $T$, the following relation should be satisfied
$$
(\mathrm{Id}-\tilde R_p(T,0)) h_p^0(x,0) = -\int_0^T \tilde R_p(T,\varphi) x \, \dot{\tilde p}_{eq}(\varphi)  d\mathrm{\varphi}
$$
and since $\mathrm{Id}-\tilde R_p(T,0)$ is invertible (because of the spectral properties of $\tilde R_p(T,0)$, detailed in Lemma $(\ref{SpecK})$), the only solution of the previous equation is given by 
$$
h_p^0(x,0) = -(\mathrm{Id}-\tilde R_p(T,0))^{-1} \int_0^T \tilde R_p(T,\varphi) x \, \dot{\tilde p}_{eq}(\varphi) d\mathrm{\varphi},
$$
leading to 
\begin{equation}
h_p^0(x,\theta) =- x \,\underbrace{\tilde R_p(\theta,0)(\mathrm{Id}-\tilde R_p(T,0))^{-1} \int_{\theta-T}^T \tilde R_p(0,\varphi) \dot{\tilde p}_{eq}(\varphi)d\mathrm{\varphi}}_{I_p(\theta)}.\end{equation}

Finally, the change of variables is obtained as 
\begin{eqnarray}
z_p(t)&=\tilde{y}_p(t)+x_p(t) \,\tilde R_p(\theta,0)(\mathrm{Id}-\tilde R_p(T,0))^{-1} \int_{\theta-T}^T \tilde R_p(0,\varphi) \dot{\tilde p}_{eq}(\varphi)d\mathrm{\varphi},\nonumber\\
 z_q(t)&=\tilde{y}_q(t)+x_q(t) \,\tilde R_q(\theta,0)(\mathrm{Id}-\tilde R_q(T,0))^{-1} \int_{\theta-T}^T \tilde R_q(0,\varphi) \dot{\tilde q}_{eq}(\varphi)d\mathrm{\varphi},\nonumber
\end{eqnarray}
and the final version of system (\ref{Main2}) as 
\begin{equation}\left\{\begin{array}{ll}
 \dot{x}_p&= f^x_2\left(x_p,x_q,z_p,z_q,\frac{t}{\ep}\right)\\
 \dot{z}_p&=\frac{1}{\ep}\tilde{K}_p\left(\frac{t}{\ep}\right) z_p+ f^y_2\left(x_p,x_q,z_p,z_q,\frac{t}{\ep}\right)\\
 \dot{x}_q&= g^x_2\left(x_p,x_q,z_p,z_q,\frac{t}{\ep}\right)\\
 \dot{z}_q&=\frac{1}{\ep}\tilde{K}_q\left(\frac{t}{\ep}\right) z_q+ g^y_2\left(x_p,x_q,z_p,z_q,\frac{t}{\ep}\right)
\label{Var1}
\end{array}\right.
\end{equation}
with the following definitions
\begin{eqnarray}
 f^x_2\left(x_p,x_q,z_p,z_q,\theta\right)=&f^x_1\left(x_p,x_q,z_p+h_p^0(x_p,\theta),z_q+h_q^0(x_q,\theta),\theta\right)\nonumber\\
  f^y_2\left(x_p,x_q,z_p,z_q,\theta\right)=& f^y_1\left(x_p,x_q,z_p+h_p^0(x_p,\theta),z_q+h_q^0(x_q,\theta),\theta\right)\nonumber\\
&+f^x_2\left(x_p,x_q,z_p,z_q,\theta\right)\tilde R_p(\theta,0)(\mathrm{Id}-\tilde R_p(T,0))^{-1} \int_{\theta-T}^T \tilde R_p(0,\varphi) \dot{\tilde p}_{eq}(\varphi)d\mathrm{\varphi}\nonumber
 \end{eqnarray}
and the equivalent definitions for $g^x_2$ and $g^y_2$. Eventually, denoting
\[X=\left(x_p,x_q\right)^T,\quad Z=\left( z_p,z_q\right)^T \]
 and $B(\theta)=\begin{pmatrix}\tilde{K}_p\left(\theta\right) & 0\\0 & \tilde{K}_q\left(\theta\right)\end{pmatrix}$, we obtain a system of the form 
\begin{equation}
\left\{\begin{array}{lll}
\dot{X}&=\Phi^x\left(X,Z,\frac{t}{\ep}\right), & X\left(t_0\right)=X_0, \\[1 mm]
\dot{Z}&=\frac{1}{\ep}B\left(\frac{t}{\ep}\right)Z+\Phi^z\left(X,Z,\frac{t}{\ep}\right), & Z\left(t_0\right)=Z_0,
\label{EqMain4}\end{array}\right.
\end{equation}
where $\Phi^x$ and $\Phi^z$ are assumed to have continuous derivatives w.r.t. $X$ and $Z$ up to order $r$ and to be periodic and continuously differentiable w.r.t. $\theta\in\T$, and do not have any prefactor in $\frac{1}{\ep}$.\\

Let $\alpha$ be positive. Since $\Phi^x$ and $\Phi^z$ are local Lipschitz, there exists a $T_\alpha>0$ such that for all $\forall t\in[0,T_\alpha], ~\|(X(t),Z(t))\|\leq\alpha$. Hence, we can work on the ball of radius $\alpha$.

{\bf Third step:} We prove that the differential system $(\ref{EqMain4})$ is equivalent to the system
\begin{equation}
\left\{
\begin{array}{lll}
\dot{X}&=F\left(X,Z,\theta\right), & X(0)=X_0, \\[1 mm]
\dot{\theta}&=\frac{1}{\varepsilon}, & \theta(0)=\theta_0:=\frac{t_0}{\eps}, \\[1 mm]
\dot{Z}&=\frac{1}{\ep}B(\theta)Z+G\left(X,Z,\theta\right), & Z(0)=Z_0,
\label{Var}
\end{array}
\right.
\end{equation}
where $F$ and $G$ have continuous derivatives w.r.t. $X$ and $Z$ up to order $r$, are periodic and continuously differentiable w.r.t. $\theta\in\T$ and, {\bf in addition}, are globally bounded and Lipschitz on $\R^n \times \R^m \times \T$. More precisely, for all $\alpha>0$, there exist functions $F$ and $G$ which coincide with $\Phi^x$ and $\Phi^z$ on the set $B_\alpha=\{ (X,Z,\theta) \in \R^n \times \R^m \times \T, \|X\| \leq \alpha \text{ and } \|Z\| \leq \alpha\}$, such that for all $(X,Z,\theta) \in \R^n \times \R^m \times \T$ and $(\tilde X,\tilde Z, \tilde \theta) \in \R^n \times \R^m \times \T$
\begin{equation}\left\{\begin{array}{ll}
\|F(X,Z,\theta)\|+\|G(X,Z,\theta)\| &\leq M \\
\|F(X,Z,\theta)-F(\tilde{X},\tilde{Z},\tilde{\theta})\|&\leq L (\|X-\tilde{X}\|+\|Z-\tilde{Z}\|+|\theta-\tilde{\theta}|)\\
\|G(X,Z,\theta)-G(\tilde{X},\tilde{Z},\tilde{\theta})\|&\leq L (\|X-\tilde{X}\|+\|Z-\tilde{Z}\|+|\theta-\tilde{\theta}|)
\label{In1}
\end{array}\right. 
\end{equation}
Indeed, consider a $\mathcal{C}^\infty$ function $\Psi:\R \mapsto [0,1]$ such that 
\begin{align*}
\Psi(a)=1 \text{ if } a \leq 1,\quad \Psi(a) = 0 \text{ if } a \geq 4,
\end{align*}
define, for a given fixed $\alpha>0$, the functions $F$ and $G$ by
\begin{align*}
F(X,Z,\theta)& =\Phi^x\left(\Psi\left( \frac{\|X\|^2}{\alpha}\right) \, X,\Psi\left( \frac{\|Z\|^2}{\alpha}\right) \, Z,\theta\right),\\
G(X,Z,\theta)&=\Phi^y\left(\Psi\left( \frac{\|X\|^2}{\alpha}\right) \, X,\Psi\left( \frac{\|Z\|^2}{\alpha}\right) \, Z,\theta\right),
\end{align*}
and assume that the norm used here is the $2$-norm. Then $F$ and $G$ retain the smoothness of $\Phi^x$ and $\Phi^y$ and it is easy to check that they are globally bounded and Lipschitz. 

\begin{remark}
System (\ref{Var}) where we have introduced explicitly $s=t-t_0$, $\theta=\frac{s}{\eps} \in \T$ and $\theta_0:=t_0/\eps$   is nothing but the autonomous form of  (\ref{EqMain4}) as long as $X$ remains in the set $B_\alpha$. Proving the existence of a center manifold for the differential system (\ref{Var}) thus automatically states the existence of a center manifold for the original system as long as $X(t)$ remains in $B_\alpha$. Now, since $F$ and $G$ are Lispchitz functions, solutions $X(t)$ and $Z(t)$ of $(\ref{Var})$ indeed exist for all times and all initial values $(X_0,Z_0,\theta_0)$. \\ \\
\end{remark}

\section{A center manifold theorem} \label{sect:manifold}
\subsection{Existence of the center manifold}

\subsubsection{Spectral properties of the resolvent}
In this subsection, we introduce a lemma, which gives us an exponential decrease of the resolvent related to the equation \begin{equation} \dot{Z}(t)=\frac{1}{\ep}B\left(\frac{t}{\ep}\right)Z+G\left(X,Z,\frac{t}{\ep}\right),  Z(0)=Z_0.\label{VarZ}\end{equation}
This is the main tool used to deal with our time dependent case as with the constant transport case presented in the article \cite{Castella}.

\begin{lemma}[Exponential decrease of the resolvent]\label{expo}
For all $t^*>0$, there exists $\mu>0$ such that, denoting $R(t,s)$ the resolvent of equation $(\ref{VarZ})$, we have 
\begin{equation} 
\forall t\in[0,t^*],~\forall s\in[0,t^*] ~\text{with}~t \geq s,\quad \left\|R(t,s)\right\| \leq C e^{-\mu(t-s)}.
\label{exp}
\end{equation}
\end{lemma}

\begin{proof}
It is an application of the Floquet theory combined with the generalized entropy method, used in the context of the Perron-Frobenius theorem,
see \cite{Perthame}.
\end{proof}

\begin{remark}
For $t\leq s$, we use the fact that $\dot R (t,s)=- B(t) R(t,s)$ to prove that:
\begin{equation} 
\exists \mu>0,~\forall t\in[0,t^*],~\forall s\in[0,t^*] ~\text{with}~s \geq t,\quad \left\|R(t,s)\right\| \leq C e^{\mu(s-t)}.
\end{equation}
\end{remark}

\subsubsection{Existence of a fast time dependent center manifold}

We are about to prove that there is a center manifold linked with the problem, which means that for all $\alpha>0$, for $\ep$ small enough (depending on $\alpha$), for $\|X(t)\|\leq \alpha$, there exists a function $h_\ep$ periodic in $\theta=\frac{t}{\ep}$ such that if $Z(0)=h_\ep(X(0),0)$, \[Z(t)=h_\ep\left(X(t),\theta\right)\] for all $t$.\\

To do so, we adapt the proof developped in \cite{Carr}.

\begin{theorem}[Existence of a fast time dependent center manifold]
Consider the differential system
\begin{equation}
\left\{\begin{array}{lll}
\dot{X}&=F\left(X,Z,\theta\right), \\[1 mm]
\dot{\theta}&=\frac{1}{\varepsilon},\\[1 mm]
\dot{Z}&=\frac{1}{\ep}B(\theta)Z+G\left(X,Z,\theta\right),
\label{Var3}\end{array}\right.
\end{equation}
where $B(\theta)$ is a periodic matrix such that, denoting $R(t,s)$ the resolvent of equation $(\ref{Var3})$, for all $t^*>0$ there exists $\mu>0$ and $\tilde\mu>0$ with: 
\begin{align*} 
\forall (t,s)\in[0,t^*]^2&~\text{with}~t \geq s,\quad \left\|R(t,s)\right\| \leq C e^{-{\mu}(t-s)}\\
&~\text{with}~t < s,\quad \left\|R(t,s)\right\| \leq C e^{{\tilde\mu}(s-t)}.
\end{align*}
In addition, assume that $F$ and $G$ are Lipschitz and have bounded  continuous derivatives w.r.t. $X$ and $Z$ up to order $r$ and are periodic and continuously differentiable w.r.t. $\theta\in\T$. Then, for all $\alpha>0$, there exists $\eps_0>0$ depending on $\alpha$ and a function $h_\ep(X,\theta)$, defined for all for all $0<\eps<\eps_0$, $\|X\|<\alpha$ and $\theta\in\T$,  where $h_\eps$ has continuous derivatives w.r.t. $x$ up to order $r$ and is continuously differentiable w.r.t. $\theta\in\T$, with the following property: for all $X_0\in\R^n$ and $\theta_0\in\T$, the solution $(X(t),\theta(t),Z(t))$ of  (\ref{Var3}) with initial conditions 
$$ X(0)=X_0, \qquad \theta(0)=\theta_0, \qquad Z(0)=h_\ep(X_0,\theta_0),$$
satisfies the relation 
$$
Z(t)=h_\ep(X(t),t/\eps). 
$$ 
The set $\{h_\ep(X,\theta), \, X\in\R^n, \, \theta\in\T\}$ is the center manifold related to our system.
\label{CenterManifold}
\end{theorem}

\begin{proof} The proof proceeds in several steps. \\ \\

\textbf{Step 1: Center manifold as the fixed-point of an operator $\cal T$.} Consider a smooth function $(X,\theta) \in \R^n \times \T \mapsto h(X,\theta)$ and initial values $(X_0,Z_0,\theta_0)  \in \R^n \times \R^m \times \T$ and denote $\theta(t,\theta_0)$, $X_h(t,X_0,\theta_0)$ and  $Z_h(t,X_0,\theta_0)$ the solution components of the differential system
 \begin{eqnarray} \label{eq:CS}
 \left\{\begin{array}{llllll}
\dot{X}_h &=& F(X_h,h(X_h,\theta),\theta), & X_h(0,X_0,\theta_0)&=& X_0,\\
\dot{\theta}&=&\frac{1}{\varepsilon}, & \theta(0,\theta_0)&=&\theta_0, \\
\dot{Z}_h&=&\frac{1}{\varepsilon}B(\theta)Z_h(t)+G(X_h,h(X_h,\theta),\theta), & Z_h(0,X_0,\theta_0)&=& Z_0.
\end{array} \right. 
\end{eqnarray}
Given that $X_h$ can be obtained independently of $Z_h$, $Z_h$ can in turn be obtained as follows: If we denote by $R(s,s_0)$ the resolvent of the differential equation
\[\left\{\begin{array}{ll}
\frac{\mathrm{d}R(s,s_0)}{\mathrm{ds}} &=B(s)R(s,s_0)\\
R(s_0,s_0)&=\mathrm{Id}
\end{array} \right. \]
and $\tilde Z_h(s,X_0,\theta_0):=Z_h(\eps s,X_0,\theta_0)$ the solution of \\
$
\dot{\tilde Z}_h(s,X_0,\theta_0) = B(\theta_0+s) \tilde Z_h(s,X_0,\theta_0) + \eps b_h(s,X_0,\theta_0)
$
with 
$$
b_h(s,X_0,\theta_0)=G(X_h(\eps s,X_0,\theta_0),h(X_h(\eps s,X_0,\theta_0),\theta_0+s),\theta_0+s)
$$
then $\tilde Z_h$ may then be written as
$$
\tilde Z_h(s,X_0,\theta_0) = R(\theta_0+s,\theta_0) Z_0 + \eps \int_{0}^{s} R(\theta_0+s,\theta_0+u) b_h(u, X_0, \theta_0) \mathrm{du}
$$
so that 
\begin{eqnarray} \label{Eqy2} \hspace{-1cm}
Z_h(t,X_0,\theta_0)
&=& R\left(\theta_0+t/\eps,\theta_0\right) \Big( Z_0
+\eps\int_{0}^{t/\eps} R(\theta_0,\theta_0+u)b_h(u, X_0, \theta_0) \mathrm{du} \Big) 
\end{eqnarray}
Now, given that $Z_h(t,X_0,\theta_0)$ must coincide with $h(X_h(t,X_0,\theta_0),\theta_0+t/\eps)$ for all values of $t$, $X_0$ and $\theta_0$, it should be in particular bounded for all times given that  $X_h(t,X_0,\theta_0)$ is and that $h(X,\theta)$ is smooth in $X$ and periodic w.r.t. $\theta$. This means that $Z_0$ can not be chosen freely but should rather be an initial value that makes $Z_h(t,X_0,\theta_0)$ bounded for all times. The only choice consists in taking 
$$
Z_0 = - \eps \lim_{t \rightarrow -\infty} \int_{0}^{t/\eps} R(\theta_0,\theta_0+u)b_h(u,X_0,\theta_0) \mathrm{du} = 
\eps \int_{-\infty}^{0} R(\theta_0,\theta_0+u)b_h(u,X_0,\theta_0) \mathrm{du}
$$
and accordingly 
\begin{eqnarray} 
Z_h(t,X_0,\theta_0) &=& \eps  \int_{-\infty}^{t/\eps} R(\theta_0+t/\eps,\theta_0+u)b_h(u,X_0,\theta_0) \mathrm{du}.\label{eq:Zh} 
\end{eqnarray}
We finally define ${\cal T} h $ as the function which maps  $(X_0,\theta_0)\in\R^n\times\T$ to
\begin{align} \label{eq:T}
({\cal T} h) (X_0,\theta_0)  =\eps \int_{-\infty}^{0} R(\theta_0,\theta_0+u)b_h(u,X_0,\theta_0) \mathrm{du}.
\end{align}
Let us show that if $h$ is a fixed point of $\cal T$, then the relation $Z_h(0,X_0,\theta_0)=h(X_0,\theta_0)$ implies that $Z_h(t,X_0,\theta_0)=h(X_h(t,X_0,\theta_0),\theta_0+t/\eps)$ for all $t$. To this aim, we thus consider fixed $X_0$ and $\theta_0$ and use the definition of ${\cal T} h$, namely
\begin{align*}
&h(X_h(t,X_0,\theta_0),\theta_0+t/\eps)\qquad\\
&\qquad\qquad=\eps \int_{-\infty}^{0} R(\theta_0 +t/\eps, \theta_0 + t/\eps + u) b_h(u,X_h(t,X_0,\theta_0),\theta_0+t/\eps)\mathrm{du}.
\end{align*}
Owing to the group law
$$
\forall (t,t'), \quad X_h(t',X_h(t,X_0,\theta_0),\theta_0+t/\eps) = X_h(t+t',X_0,\theta_0),
$$
we have
$$
b_h(u,X_h(t,X_0,\theta_0),\theta_0+t/\eps) = b_h(u+t/\eps,X_0,\theta_0),
$$
which leads to 
\begin{align*}
h(X_h(t,X_0,\theta_0),\theta_0+t/\eps)&=\eps \int_{-\infty}^{0} R(\theta_0 +t/\eps, \theta_0 + t/\eps + u) b_h(u+t/\eps,X_0,\theta_0)\mathrm{du} \\
&=\eps \int_{-\infty}^{t/\eps} R(\theta_0 +t/\eps, \theta_0 + u) b_h(u,X_0,\theta_0)\mathrm{du} \\
&= Z_h(t,X_0,\theta_0)
\end{align*}
where the last equality follows from (\ref{eq:Zh}). \\ \\
\textbf{Step 2: $\mathcal{T}$ maps $\mathcal{F}$ to $\mathcal{F}$.} 
Define $\mathcal{F}$ as the functional space
$$
\mathcal{F}=\{h\in {\cal C}^1(\R^n\times\T, \R^{m}),~\text{such that}~\|h\|_{\infty}\leq \alpha~\text{and}~ \|\partial_x h\|_{\infty}\leq 1\}
$$
where $\|\partial_x h\|_{\infty}=\|\partial_x h\|_{L^\infty\left(\R^n\times\T,\mathcal{L}\left(\R^n,\R^{m}\right)\right)}$. We wish to show now that $\cal T$ maps $\cal F$ to itself: given $h \in {\cal F}$ and the definition (see (\ref{eq:T})) of ${\cal T}h$, we have for all $(X_0,\theta_0)\in\R^n\times\T$
\begin{align*}
\|\mathcal{T}h(X_0,\theta_0)\|&\leq \eps \int_{-\infty}^{0} \|R(\theta_0,\theta_0 + u)\| \|b_h(u, X_0, \theta_0) \|  \mathrm{du} \\
&\leq\ep  \int_{-\infty}^{0} C e^{\beta  u} \| G(X_h(\ep s,X_0,\theta_0),h(X_h(\ep s,X_0,\theta_0),\theta_0+u),\theta_0+u)\|\mathrm{du}.
\end{align*}
According to $(\ref{In1})$, we have $\|F(X_h,h(X_h,\theta),\theta)\| \leq M$, so that 
\begin{align*}
\|\mathcal{T}h(X_0,\theta_0)\| &\leq \ep \int_{-\infty}^{0} C e^{\beta  u}  M \mathrm{du} = C \eps \frac{ M}{\beta} \leq \alpha 
\end{align*}
provided $\eps < \eps_1:= \frac{\alpha \, \beta}{C \, M}$. Hence, for $\eps <\eps_3$, $\|{\cal T} h\|_\infty \leq \alpha$. Now, $h$ and $F$ are periodic w.r.t. $\theta_0$, so that $X_h$ and $b$ are periodic as well and since $R(\theta_0,\theta_0+u)$ is, ${\cal T} h$ is clearly periodic w.r.t. $\theta_0$. It remains to prove that $\|\partial_{X} (\mathcal{T}h)\|_\infty\leq 1$. To this aim, we first estimate $\partial_{X_0} X_h(t,X_0,\theta_0)$ from the variational equation 
\begin{align*}
\left(\partial_{X_0} X_h \right)&(t,X_0,\theta_0)={\rm Id} + \int_0^t \Big(\partial_{X}F(X_h,h(X_h,\theta_0+s/\eps),\theta_0+s/\eps)) \cdot \partial_{X_0} X_h \\
&+\partial_{Z}F(X_h,h(X_h,\theta_0+s/\eps),\theta_0+t/\eps) \cdot (\partial_{X} h)(X_h,\theta_0+s/\eps) \cdot \partial_{X_0} X_h \Big) ds
\end{align*}
as follows 
\begin{align*}
\left\| \left(\partial_{X_0} X_h\right) \right\| \leq 1+ 2 L \int_0^{|t|} \|\partial_{X_0} X_h\| ds,
\end{align*}
which, owing to Gronwall lemma, leads to 
\begin{equation}
\forall t\in\R,~\|\partial_{x_0} X_h(t,.,.) \| \leq e^{2L |t|}.
\label{Step5x}
\end{equation}
Substituting this estimate into the equation obtained by differentiating ${\cal T}h$ we get
\begin{align*}
\|  \partial_{X_0} \mathcal{T}h(X_0,\theta_0)\| &\leq C \eps \int_{-\infty}^{0} e^{\beta u} \| \partial_X G \cdot \partial_{X_0} X_h +\partial_Z G  \cdot \partial_{X} h \cdot \partial_{X_0} X_h\| \mathrm{du}
\end{align*}
where all arguments of $h$, $X_h$, $F$ and $G$ are as in (\ref{eq:T}) and have been omitted for the sake of clarity. Using (\ref{Step5x}) it then follows that 
\begin{align*}
\|  \partial_{X_0} \mathcal{T}h(X_0,\theta_0)\| &\leq 2 \,C \, \eps L  \int_{-\infty}^{0} e^{\beta u} e^{2  \eps  L |u|}\mathrm{du} = \frac{2 C \eps L}{\beta-2 \eps L}
\end{align*}
provided $\beta-2 \eps L >0$ and this last term is less than $1$ for $\eps < \eps_2:= \frac{\beta}{(2C+1)L}$. \\ \\
\textbf{Step 3: $\mathcal{T}$ is a contraction.} Consider $h_1$ and $h_2$ two functions of $\mathcal{F}$. The corresponding functions $X_{h_1}$ and $X_{h_2}$ satisfy
$$
\|(X_{h_1}-X_{h_2})(t,X_0,\theta_0)\| \leq L \int_{0}^t  \Big( \|X_{h_1}-X_{h_2}\| + \|h_1(X_{h_1},\theta_0+u)-h_2(X_{h_2},\theta_0+u) \|\Big) du
$$
where, once again, the arguments $(u,X_0,\theta_0)$ of $X_{h_1}$ and $X_{h_2}$ on the r.h.s have been omitted for brevity. It is straightforward to write, say with $\theta=\theta_0+u$, that 
\begin{align*}
\|h_1(X_{h_1},\theta)-h_2(X_{h_2},\theta) \| &\leq \|h_1(X_{h_1},\theta)-h_1(X_{h_2},\theta) \|
+\|h_1(X_{h_2},\theta)-h_2(X_{h_2},\theta) \|\\
&\leq \|X_{h_1}-X_{h_2}\|+\|h_1-h_2\|_{\infty}.
\end{align*}
Hence, 
$$
\|(X_{h_1}-X_{h_2})(t,X_0,\theta_0)\| \leq  \Big| \int_{0}^t  L \|X_{h_1}-X_{h_2}\| du \Big| + |t| L \|h_1-h_2\|_{\infty}
$$
and by Gronwall lemma, we obtain 
\begin{equation*}
\left\|X_{h_1}(t,X_0,\theta_0)-X_{h_2}(t,X_0,\theta_0)\right\|_\infty\leq L |t| e^{L |t|} \|h_1-h_2\|_\infty.
\end{equation*}
Consequently, we have 
\begin{align}
\|\mathcal{T}h_1(X_0,\theta_0)&-\mathcal{T}h_2(X_0,\theta_0)\| \leq \eps L \int_{-\infty}^{0} e^{\beta u}  \Big(\|X_{h_1}(\eps u,X_0,\theta_0)-X_{h_2}(\eps u,X_0,\theta_0)\| \nonumber \\
&+\|h_1(X_{h_1}(\eps u,X_0,\theta_0),\theta_0+u)-h_2(X_{h_2}(\eps u,X_0,\theta_0),\theta_0+u)\|\Big)\mathrm{du} \nonumber \\
&\leq \eps L \|h_1-h_2\|_\infty \int_{-\infty}^{0}  e^{\beta u} \Big(2 \eps |u| L e^{\eps L |u|} +1\Big)\mathrm{du} \nonumber \\
&\leq \Big(\frac{\eps L}{\beta}+\frac{\eps^2 L^2}{(\beta-\eps L)^2} \Big) \|h_1-h_2\|_\infty
\label{Step6x}
\end{align}
so that $\mathcal{T}:\mathcal{F}\to \mathcal{F}$ becomes a contraction for small enough values of $\eps$. \\ \\
\textbf{Step 4: Smoothness of $h$.} The idea is to repeat the proof of Step 4. within the set 
$$
{\mathcal F}^r =\{h\in {\cal C}^r(\R^n\times\T, \R^{m}),~\text{such that}~\|h\|_{\infty}\leq \alpha~\text{and}~ \|\partial_x^k h\|_{\infty}\leq 1 \text{ for all } k=1,\ldots,r\}
$$
Since all derivatives up to order $r$ of $F$ and $G$ are bounded, inequality (\ref{Step6x}) is simply replaced by \begin{align*}
\| \partial_{x}^k( \mathcal{T}h_1(X_0,\theta_0)&- \mathcal{T}h_2(X_0,\theta_0))\| \\
&\leq C \,\eps \, \left(\|h_1-h_2\|_\infty +\|\partial_x(h_1-h_2)\|_\infty +\ldots  + \|\partial_x^k(h_1-h_2)\|_\infty \right)
\end{align*}
where $C$ is a constant depending on $k$, $\alpha$ and $\beta$ for small enough $\eps$ and 
where the norm used is the induced norm on $k$-linear functions. By choosing $\eps$ small enough, we again obtain a contraction map. The smoothness of $h$ in $x$ thus follows. By definition, it is obviously $C^1$ w.r.t. $\theta$. 
\end{proof}

\begin{remark}
The function $h$ also depends smoothly on $\eps$, as it is obtained as the limit of the convergent iteration 
$h=\lim_{k \rightarrow \infty} {\cal T}^k h_0$ from a $\eps$-independent $h_0$. It is thus $C^\infty$ w.r.t. $\eps$. 
\end{remark}

We have proved the existence of a center manifold. Now, we want to prove that when we do not have initial conditions such that $Z_0=h(X_0,0)$, the exact solution of the differential system $(\ref{VarRef})$ goes exponentially fast to the center manifold.

\begin{theorem}[Error relative to the center manifold]
\label{Error}
Denote $X(t)$ and $Z(t)$ the solutions of system (\ref{Var3}) with prescribed initial values. Under the assumptions of Theorem \ref{CenterManifold}, the following assertions hold true:
\begin{enumerate}
\item {\bf Exponential convergence towards the center manifold: } There exist strictly positive constants $C$ and $\tilde \mu$ such that
$$
\forall t\geq 0,\quad\left\|Z(t) -h_\ep\left(X(t),\theta_0+\frac{t}{\varepsilon}\right)\right\|\leq C e^{-\tilde \mu \frac{t}{\ep}}.
$$

\item {\bf Shadowing principle for the complete system:}  There exists a constant $C>0$ , independent of $\eps$ and $t^*$, such that for any $t^*\geq 0$, there exists an altered initial data $X_0^\ep$ (implicitly depending on $t^*$), such that the solution components of the reduced system
\[\left\{\begin{array}{ll}
\frac{\mathrm{d} X_h}{\mathrm{dt}}&=F\left(X_h,h_\ep\left(X_h,\theta_0+\frac{t}{\varepsilon}\right),\theta_0+\frac{t}{\eps}\right)\\
X_h(0)&=X_0^\eps\\
Z_h(t) &=h_\ep\left(X_h(t),\theta_0+\frac{t}{\varepsilon}\right)
\end{array}\right. \]
satisfy the following error estimate on $[0,t^*]$
$$
\forall t\in [0,t^*],\quad\|Z(t) -Z_h(t)\|+\|X(t) -X_h(t)\| \leq C e^{-\hat\mu \frac{t}{\ep}}.
$$
Moreover, if the solution $X$ is bounded on $\R^+$, we can take $t^*=+\infty$ in the above estimates.
\end{enumerate}
\end{theorem}

\begin{proof} By construction of the function $h_\eps$, it satisfies for all $X \in \R^n$ and for all $\theta \in \T$
$$
\frac{1}{\eps} B(\theta) h_\eps(X,\theta) + G(X,h_\eps(X,\theta),\theta) = \frac{1}{\eps} \partial_\theta h_\eps(X,\theta) + \partial_X h_\eps(X,\theta) F(X,h_\eps(X,\theta),\theta).
$$
Hence, 
\begin{align*}
&\hspace{-0.5cm}\frac{\mathrm{d}h_\eps\left(X(t),\theta_0+\frac{t}{\ep}\right)}{\mathrm{dt}}\\
&=\frac{1}{\eps} B\left(\theta_0+\frac{t}{\eps}\right) h_\eps \left(X(t), \theta_0+\frac{t}{\eps}\right)+G\left(X(t),h_\eps\left(X(t), \theta_0+\frac{t}{\ep}\right),\theta_0+\frac{t}{\ep}\right). 
\end{align*}
The Duhamel formula then leads to 
\begin{align*}
Z(t)-h_\eps\left(X(t),\theta_0+\frac{t}{\eps}\right)&= R\left(\theta_0 +\frac{t}{\ep},\theta_0\right)\left(Z_0-h_\eps\left(X_0,\theta_0\right)\right)\\
&\quad\quad+\eps \int_0^{\frac{t}{\ep}}R\left(\theta_0 +\frac{t}{\ep},\theta_0 +u\right) \Delta G(\eps u) \mathrm{du}
\end{align*}
with 
$$
(\Delta G)(\eps u):= G(X(\eps u), Z(\eps u),\theta_0+u)-G(X(\eps u),h_\eps\left(X(\eps u),\theta_0+u\right),\theta_0+u).
$$
We then obtain the following inequality 
\begin{align*}
\left\|Z(t)-h_\eps\left(X(t),\theta_0+\frac{t}{\ep}\right)\right\|&\leq C e^{-\mu \frac{t}{\ep}}\left\|Z_0-h_\eps (X_0,\theta_0))\right\| \\
&\qquad+C \ep \int_0^{\frac{t}{\ep}} e^{-\mu\left(\frac{t}{\ep}-u\right)} \left\|(\Delta G)(\eps u)\right\|\mathrm{du}
\end{align*}
with 
$$
\|(\Delta G)(\eps u) \| \leq L  \|Z(\eps u)-h_\eps(X(\eps u),\theta_0+u) \|.
$$
Denoting 
$$
r(t) = e^{\frac{\mu}{\eps} t} \|Z(t)-h_\eps (X(t),\theta_0+t/\eps)\|,
$$ 
we thus have 
$$
r(t) \leq C r(0) + C L  \int_0^{t}  r(u) \mathrm{du}
$$
and upon using Gronwall lemma we obtain $r(t) \leq C r(0) e^{C L t}$. Going back to the quantity $ e^{\frac{\mu}{\eps} t} r(t)$ we finally get
\begin{align*}
\left\| Z(t)-h\left(X(t),\frac{t}{\ep}\right)\right\|&\leq C e^{(C L-\frac{\mu}{\ep}) t} \|Z_0-h_\eps(X_0,\theta_0)\| 
\end{align*}
and the first statement follows with $0<\tilde \mu < \mu -\ep C L$ for $0 < \eps < \frac{\mu}{C L}$. \\
Consider $t^* > 0$ and denote $X^*=X(t^*)$ the values of the solution of (\ref{Var3}) with initial conditions $(X,Z,\theta)(0) = (X_0,Z_0,\theta_0)$. The function $X \mapsto F(X,h_\eps(X,\theta),\theta)$ being a Lipschitz function w.r.t. $X$, the following system  
\begin{equation} \label{eq:sr}
\left\{\begin{array}{ll}
\dot{X}_h&=F\left(X_h,h_\eps \left(X_h,\theta_0+\frac{t}{\eps}\right),\theta_0 + \frac{t}{\ep}\right)\\
X_{h}(t^*)&=X^*
\end{array}\right. 
\end{equation}
has a unique solution on the interval $[0,t^*]$ so that we may consider $X^*_0:= X_h(0)$. Now, function $X(t)$ satisfies the differential equation 
$$
\dot X(t) = F(X(t),h_\eps(X(t),\theta_0+t/\eps),\theta_0+t/\eps) + \delta(t)
$$
where 
$$
\delta(t) = F(X(t),Z(t),\theta_0+t/\eps) - F(X(t),h_\eps(X(t),\theta_0+t/\eps),\theta_0+t/\eps)
$$
may be bounded, according to the first statement of this theorem, as follows
$$
\|\delta(t) \| \leq L \|Z(t)-h_\eps(X(t),\theta_0+t/\eps)\| \leq C L e^{-\tilde \mu \frac{t}{\ep}}. 
$$
We then have 
\begin{equation*}
\dot X  - \dot X_h =F(X,h_\eps(X, \theta_0+t/\eps),\theta_0+t/\eps)-F(X_h,h_\eps(X_h, \theta_0+t/\eps),\theta_0+t/\eps) + \delta(t)
\end{equation*}
with $X(t^*)-X_h(t^*)=0$. Integrating backward from $t^*$ to $t$ for $t\leq t^*$ and taking norms of both sides, it follows that 
\begin{align*}
\|X(t)-X_h(t)\|  & \leq \int_{t}^{t^*} L \Big( (1+\|\partial_X h_\eps\|_{\infty}) \left\| X(u)-X_h(u)\right\| + C e^{-\tilde \mu \frac{u}{\ep}}  \Big) \mathrm{du} \\
& \leq \int_{t}^{t^*} \beta \left\| X(u)-X_h(u)\right\| \mathrm{du} + \frac{\ep C L}{\tilde \mu} \left(e^{- \tilde \mu \frac{t}{\ep}}-e^{- \tilde \mu \frac{t^*}{\ep}}\right)
\end{align*}
with $\beta:=L  (1+\|\partial_X h_\eps\|_{\infty})$. We now apply Gronwall lemma and get
\begin{align*}
\|X(t)-X_h(t)\|  & \leq \frac{\ep C L}{\tilde \mu} \left(e^{- \tilde \mu \frac{t}{\ep}}-e^{- \tilde \mu \frac{t^*}{\ep}}\right) +  \frac{\ep C L}{\tilde \mu} \beta \int_{t}^{t^*} \left(e^{- \tilde \mu \frac{s}{\ep}}-e^{- \tilde \mu \frac{t^*}{\ep}}\right) e^{\beta (s-t)}\mathrm{ds}\\
& \leq \frac{\ep C L}{\tilde \mu} \left( 1-\frac{\beta}{\beta-\frac{\tilde\mu}{\ep}} \right)e^{-\tilde\mu \frac{t}{\ep}} + \frac{\ep C L}{\tilde \mu}  \beta\left( \frac{1}{\beta-\frac{\tilde\mu}{\ep}}-\frac{1}{\beta} \right)e^{-\beta \frac{t}{\ep}+(\beta-\tilde\mu)\frac{t^*}{\ep}}
\end{align*}
and as $\tilde\mu$ goes to $+\infty  $ as $\ep$ goes to $0$, we have $\left( \frac{1}{\beta-\frac{\tilde\mu}{\ep}}-\frac{1}{\beta} \right) <0$ for $\ep$ small enough. Hence, we get:
\[  \|X(t)-X_h(t)\| \leq \frac{\ep C L}{\tilde \mu} \left( 1-\frac{\beta}{\beta-\frac{\tilde\mu}{\ep}} \right)e^{-\tilde\mu \frac{t}{\ep}}. \]
\end{proof}

\subsection{Approximation of the center manifold}
In this section, we aim at showing that $h_\eps$ can be expanded in powers of $\eps$ up to every order $k \leq r$, where each coefficient-function can be computed explicitly through a recursive relation.
\begin{theorem}[Approximation of $h_\eps$] Under the assumptions of Theorem \ref{CenterManifold}, the following statements hold true: \\ 
{\bf 1.} The function $h_\eps$ satisfies  the following partial differential equation for all $X \in \R^n$ and all $\theta \in \T$
\begin{equation}
\frac{1}{\eps}\Big( \partial_\theta h_\ep(X,\theta) - B(\theta) h_\eps(X,\theta) \Big) = G(X,h_\eps(X,\theta),\theta)- \partial_X h_\eps(X,\theta) F\left(X,h_\eps(X,\theta),\theta\right).
\label{pde}
\end{equation}
{\bf 2.} The terms of the formal expansion $h_\eps = \eps h^1 + \eps^2 h^2 + \cdots$ of $h_\eps$ are defined in a unique way by an equation of the form 
\begin{equation*} 
\forall n\in\N,\quad B(\theta) h^{n+1}-\partial_\theta h^{n+1}=J_n(X,\theta)
\label{Eqhn}
\end{equation*}
where $J_n$ depends only on derivatives of $F$ and $G$ up to order $n$. Furthermore, the function 
$\tilde h_\eps:= \eps h^1 + \eps^2 h^2 + \cdots + \eps^r h^r$ satisfies equation (\ref{pde}) up to an error term of size 
$\eps^{r-1}$ and one has the following estimate for some positive constant $C_r$
\begin{equation} \label{eq:norminf}
\|h_\eps-\tilde h_\eps\|_\infty\leq C_r \, \eps^{r}.
\end{equation}
\label{Computeh}
\end{theorem}

\begin{proof} By construction, function $Z(t)=h_\ep(X(t),\theta(t))$ satisfies equation $(\ref{Var})$, i.e.\begin{align*}
\frac{\mathrm{d}h_\ep(X(t),\theta(t))}{\mathrm{dt}}&=\frac{1}{\eps}B(\theta(t)) h_\eps(X(t),\theta(t))+G(X,h_\ep(X(t),\theta(t)),\theta(t))\\
&= \partial_X h_\ep(X(t),\theta(t)) F\left(X(t),h_\eps(X,\theta(t)),\theta(t)\right)+\frac{1}{\varepsilon}\partial_\theta h_\ep(X(t),\theta(t))
\end{align*}
with $\theta(t)=\theta_0+t/\eps$. In particular, for $t=0$ we get equation (\ref{pde}) with $X=X_0$ and $\theta=\theta_0$ and since the initial values $X_0$ and $\theta_0$ are arbitrary, this proves the first statement. We now look for an expansion of $h_\eps$ in powers of $\eps$ of the form 
$$
h_\eps(X,\theta)=h^0(X,\theta)+\varepsilon h^1(X,\theta)+\dots+\varepsilon^n h^n(X,\theta)+ \cdots
$$
and thus insert previous expression into equation $(\ref{pde})$ to equate like powers of $\eps$. At order $\eps^{-1}$, this gives
$$
B(\theta) h^0 (X,\theta) = \partial_\theta h^0(X,\theta).
$$
This is an homogeneous linear differential equation in $\theta$, whose solution can be expressed as
$$
h^0(\cdot,\theta) = R(\theta,0) h^0(\cdot,0).
$$
The initial condition $h^0(\cdot,0)$ is a priori not prescribed. However, the only choice leading to a periodic solution $h^0$ is $h^0(\cdot,0)=0$, as is induced by the estimate $\|R(\theta,0)\| \leq C e^{-\mu \theta}$. Hence, $h^0 \equiv 0$.  We then proceed  to derive the equation satisfied by $h^1$, that is to say
$$
B(\theta) h^1 (X,\theta) + G(X,0,\theta) =  \partial_\theta h^1(X,\theta). 
$$
Similarly, the solution can be obtained easily 
$$
h^1(\cdot,\theta) = R(\theta,0) h^1(\cdot,0) + \int_0^\theta R(\theta,\varphi) G(X,0,\varphi)  d\mathrm{\varphi}
$$
For $h^1$ to be periodic with period $T$, the following relation should be satisfied
$$
(\mathrm{Id}-R(T,0)) h^1(\cdot,0) = \int_0^T R(T,\varphi) G(X,0,\varphi)  d\mathrm{\varphi}
$$
and since $\mathrm{Id}-R(T,0)$ is invertible, the only solution of the previous equation is given by 
$$
h^1(\cdot,0) = (\mathrm{Id}-R(T,0))^{-1} \int_0^T R(T,\varphi) G(X,0,\varphi)  d\mathrm{\varphi},
$$
leading to 
$$
h^1(\cdot,\theta) = R(\theta,0)(\mathrm{Id}-R(T,0))^{-1} \int_{\theta-T}^T R(0,\varphi) G(X,0,\varphi)  d\mathrm{\varphi}
$$
More generally, $h^{n+1}$ satisfies an equation of the form 
\begin{equation} 
B(\theta) h^{n+1}(X,\theta)-\partial_\theta h^{n+1}(X,\theta)=J_n(X,\theta)
\label{Eqhn}
\end{equation}
where $J_n$ contains various derivatives of $F$ and $G$ up to order $n$ and is periodic w.r.t. $\theta$.  The same arguments as above allow to conclude that it has a unique periodic solution
$$
h^{n+1}(\cdot,\theta) = R(\theta,0)(\mathrm{Id}-R(T,0))^{-1} \int_{\theta-T}^T R(0,\varphi) J_n(X,\varphi)  d\mathrm{\varphi},
$$
which, given the assumptions on $F$ and $G$, is bounded and has bounded derivatives w.r.t. $X$ up to order $r-n$. 
Consider now the truncated expansion $\tilde h_\eps =\eps h^1 + \ldots + \eps^r h^r$ of $h_\eps$ and denote 
$\Delta h_\eps = h_\eps - \tilde h_\eps$. Function $\tilde h_\eps$ satisfies the partial differential equation (\ref{pde}) 
\begin{align*}
\frac{1}{\eps}B(\theta) \tilde h_\eps(X,\theta)+G(X,\tilde h_\eps(X,\theta),\theta)=& \partial_X \tilde h_\eps(X,\theta) F\left(X,\tilde h_\eps(X,\theta),\theta\right)+\frac{1}{\eps}\partial_\theta h_\ep(X,\theta) \\
&\qquad+ \delta(X,\theta)
\label{Eqh}
\end{align*}
up to a defect $\delta(X,\theta)$ which is a continuous function from $\R^n \times \T$ into $\R^m$ and is bounded by construction by $\eps^{r-1}$. The solution $X_{h}(t,X_0,\theta_0)$ of equation (\ref{eq:CS}) thus satisfies
$$
\frac{d \tilde h(X_{h},\theta(t))}{dt} = \frac{1}{\eps} B(\theta(t)) \tilde h(X_{h},\theta(t)) + G(X_{h},\tilde h(X_{h},\theta(t)),\theta(t)) +\delta(X_{h},\theta(t))
$$ 
where we have omitted the arguments $(t,X_0,\theta_0)$ of $X_{h}$ for brevity. Proceeding as in Theorem $\ref{CenterManifold}$ (both $h$ and $\tilde h$ are bounded by construction), we then get 
$$
\Delta h (X_0,\theta_0) = \eps \int_{-\infty}^0 R(\theta_0,\theta_0+u) \Big(\Delta G(\eps u,X_0,\theta_0) +\delta(X_{h}(\eps u,X_0,\theta_0),\theta_0+u) \Big) d \mathrm{u}
$$
with
\begin{align*}
\Delta G(\eps u,X_0,\theta_0) &= G(X_h(\eps u,X_0,\theta_0),h(X_h(\eps u,X_0,\theta_0),\theta_0+u),\theta_0+u)\\
& -G(X_h(\eps u,X_0,\theta_0),\tilde h(X_h(\eps u,X_0,\theta_0),\theta_0+u),\theta_0+u).
\end{align*}
It follows that 
$$
\|\Delta h (X_0,\theta_0)\| \leq \eps C \int_{-\infty}^0 e^{\mu u} (L  \|\Delta h\|_{\infty} + K \eps^{r-1}) \mathrm{du}
$$
and the second statement follows. 
\end{proof}

\begin{theorem}[Shadowing principle for the truncated system]
\label{Errorhn}
Let $t^*$ and $X_0^\ep$ be as in Theorem \ref{Error} and define $X_{\tilde h}$ as the solution of differential system
\[\left\{\begin{array}{ll}
\frac{\mathrm{d} X_{\tilde h}}{\mathrm{dt}}&=F\left(X_{\tilde h},\tilde h_\eps\left(X_{\tilde h},\frac{t}{\eps}\right),\frac{t}{\varepsilon}\right)\\
X_{\tilde h}(0)&=X_0^\eps
\end{array}\right. \]
and $Z_{\tilde h}(t)=\tilde h_\eps\left(X_{\tilde h}(t),\frac{t}{\varepsilon}\right)$. Then, we have the following estimates:
$$
\forall t\in [0,t^*],\quad\|Z(t) -Z_{\tilde h}(t)\|+\|X(t) -X_{\tilde h}(t)\| \leq C \Big( \ep^{r+1}+ e^{-\hat\mu \frac{t}{\ep}}\Big),
$$
with $C>0$ and $\hat\mu>0$ constants independent of $t$ and $\eps$. Moreover, if the solution $X$ is bounded on $\R^+$, we can take $t^*=+\infty$ in the above estimates.
\end{theorem}
\begin{proof}
The results follow directly from Theorem \ref{Error} and from estimate (\ref{eq:norminf}).
\end{proof}

\subsection{Derivation of the first terms of the expansion}
In this subsection, we derive the explicit expressions of the first terms of the expansion of $h_\eps$ previously obtained from the equation 
\begin{align*}
\partial_X h_\eps(X,\theta) F\left(X,h_\eps(X,\theta),\theta\right) +\frac{1}{\ep}\partial_\theta h_\eps(X,\theta) 
=\frac{1}{\eps}B(\theta) h_\ep(X,\theta)+ G\left(X,h_\eps(X,\theta),\theta\right).
\end{align*}
More precisely and assuming that $F$ and $G$ have bounded derivatives at least up to order $r=2$, we look for the truncated expansion
$h_{[2]}=h^0+\ep h^1+\ep^2 h^2$. We have already shown in the course of Theorem \ref{Computeh} that $h^0 \equiv 0$ and that $h^1$ is given by the equation 
\begin{eqnarray}
h^1(\cdot,\theta) =& R(\theta,0)(\mathrm{Id}-R(T,0))^{-1} \int_0^T R(T,\varphi) G(X,0,\varphi)  d\mathrm{\varphi}\\
&+\int_0^\theta R(\theta,\varphi) G(X,0,\varphi)  d\mathrm{\varphi}, \nonumber \\
=& R(\theta,0)(\mathrm{Id}-R(T,0))^{-1} \int_{\theta-T}^{\theta} R(0,\varphi) G(X,0,\varphi)  d\mathrm{\varphi}.
\label{Ordre0}
\end{eqnarray}
Now, the equation at order $1$ in $\ep$ gives
\begin{equation*}
\partial_\theta h^2=B(\theta) h^2+\partial_Z G(X,0,\theta) \cdot h^1 - \partial_X h^1 \cdot F(X,0,\theta)
\end{equation*}
so that 
$$
h^2(\cdot,\theta) = R(\theta,0)(\mathrm{Id}-R(T,0))^{-1} \int_{\theta-T}^{\theta} R(0,\varphi) {\cal RHS}(X,\varphi) d\mathrm{\varphi}
$$
where we have denoted
$$
{\cal RHS}(X,\theta) = \partial_Z G(X,0,\theta) \cdot h^1 - \partial_X h^1 \cdot F(X,0,\theta).
$$

\begin{remark}
For the sake of illustration, we give the first orders reduction of the differential system $(\ref{Var3})$.
\begin{itemize}
\item The $0$ order reduction of the differential system $(\ref{Var3})$ is:
\begin{equation*}\left\{\begin{array}{ll}
 \dot{X}(t)&=F(X,0,\theta)\\
Z(t)&=0
\end{array}\right.
\end{equation*}
\item The first order reduction of our system is given by the following equations:
\begin{equation*}\left\{\begin{array}{ll}
 \dot{X}(t)&=F\left(X,0,\frac{t}{\ep}\right)+\ep \partial_z F\left(X,0,\frac{t}{\ep}\right) h^1\left(X,\frac{t}{\ep}\right)\\
&=F\left(X,0,\frac{t}{\ep}\right)\\
&+\ep \partial_z F\left(X,0,\frac{t}{\ep}\right) R\left(\frac{t}{\ep},0\right)\left(\mathrm{Id}-R(T,0)\right)^{-1} \int_{\frac{t}{\ep}-T}^{\frac{t}{\ep}} R(0,\varphi) G(X,0,\varphi)  d\mathrm{\varphi} \\
Z(t)&=R\left(\frac{t}{\ep},0\right)\left(\mathrm{Id}-R(T,0)\right)^{-1} \int_{\frac{t}{\ep}-T}^{\frac{t}{\ep}} R(0,\varphi) G(X,0,\varphi)  d\mathrm{\varphi}
\end{array}\right.
\end{equation*}
\end{itemize}

\end{remark}

\begin{remark} 
In terms of our initial system, we have the following first order reduction of the differential system $(\ref{Main3})$.
\begin{itemize}
\item The $0^{\text{th}}$ order reduction of the differential system $(\ref{Main3})$ is:
\begin{equation}\left\{\begin{array}{ll}
 \dot{x}_p&=A_p^0 x_p(t) -  B_p^0 x_p(t) x_q(t)\\
\dot{x}_q&=-A_q^0 x_q(t) + B_q^0 x_p(t) x_q(t) \\
y_p(t)&=\left( p_{eq}(\theta)+I_p(\theta)\right) x_p(t)\\
y_q(t)&=\left( q_{eq}(\theta)+I_q(\theta)\right) x_q(t)
\end{array}\right.
\end{equation}
with the following coefficients:
\begin{align*}
A_p^0&=\sum\limits_{i=1}^N a_{p,i}\left( p_{eq,i}(\theta)-I_{p,i}(\theta)\right)\\
B_p^0&=\sum\limits_{i=1}^N b_{p,i}\left( p_{eq,i}(\theta)-I_{p,i}(\theta)\right)\left( q_{eq,i}(\theta)-I_{q,i}(\theta)\right)\\
A_q^0&=\sum\limits_{i=1}^N a_{q,i}\left( q_{eq,i}(\theta)-I_{q,i}(\theta)\right)\\
B_q^0&=\sum\limits_{i=1}^N b_{q,i}\left( p_{eq,i}(\theta)-I_{p,i}(\theta)\right)\left( q_{eq,i}(\theta)-I_{q,i}(\theta)\right)
\end{align*}
This system is still in a Lotka-Volterra form, but its coefficients are some averaging (in $i$, the different sites) of the previous coefficients.

\item The first order reduction of our system is given by the following equations:
\begin{equation}\left\{\begin{array}{ll}
\dot{x}_p&=\left( A_p^0+\ep A_1^p\right) x_p(t) - \left( B_p^0+\ep B_p^1\right) x_p(t) x_q(t) -\ep C_p^1 x_p(t)^2 x_q(t) \\
&\qquad-\ep D_p^1 x_p(t) x_q(t)^2\\
\dot{x}_q&=\left( A_q^0+\ep A_1^q\right) x_q(t) - \left( B_q^0+\ep B_q^1\right) x_p(t) x_q(t) -\ep C_q^1 x_q(t)^2 x_p(t) \\
&\qquad-\ep D_q^1 x_q(t) x_p(t)^2\\
y_p(t)&=-I_p(\theta)x_p(t) - \ep \tilde{K}_p(\theta)^{-1}\left[ x_p(t) \Pi_p\left(a_p p_{eq}\right)-x_p(t) x_q(t) \Pi_p\left(b_p p_{eq}q_{eq}\right)\right]\\
y_q(t)&=-I_q(\theta)x_q(t) - \ep \tilde{K}_q(\theta)^{-1}\left[ x_q(t) \Pi_q\left(a_q q_{eq}\right)-x_p(t) x_q(t) \Pi_q\left(b_q p_{eq}q_{eq}\right)\right]
\end{array}\right.
\end{equation}
with the following coefficients:
\begin{align*}
A_p^1&=-\sum a_{p} \tilde{K}_p(\theta)^{-1} \Pi_p\left(a_p p_{eq}\right)\\
B_p^1&=-\sum b_{p} \left( p_{eq}(\theta)-I_{p}(\theta)\right) \tilde{K}_q(\theta)^{-1} \Pi_q\left(a_q q_{eq}\right)\\
&-\sum b_{q} \left( q_{eq}(\theta)-I_{q}(\theta)\right) \tilde{K}_p(\theta)^{-1} \Pi_p\left(a_p p_{eq}\right)- \sum a_p \tilde{K}_p(\theta)^{-1} \Pi_p\left(b_p p_{eq}q_{eq}\right)\\
C_p^1&=\sum b_p \left( p_{eq}(\theta)-I_{p}(\theta)\right) \tilde{K}_q(\theta)^{-1} \Pi_q\left(b_q p_{eq}q_{eq}\right) \\
D_p^1&=\sum b_p \left( q_{eq}(\theta)-I_{q}(\theta)\right) \tilde{K}_p(\theta)^{-1} \Pi_p\left(b_p p_{eq}q_{eq}\right)
\end{align*}
and the same notations for the coefficients in $q$, where we have introduced $\Pi_p$ the projection on $\mathcal{E}_0$ parallel to the direction $p_{eq}(\theta)$, and $\Pi_q$ the projection on $\mathcal{E}_0$ parallel to the direction $q_{eq}(\theta)$,.
\end{itemize}
\end{remark}

\begin{remark} $ $ \\
Here, we have written the explicit differential system for an approximation of the exact solution of $(\ref{Main3})$ up to order $\mathcal{O}\left( e^{-\hat\mu \frac{t}{\ep}} +\ep \right)$ for a constant $\hat\mu>0$.
\end{remark}

\section{Averaging} \label{sect:averaging}
\subsection{The averaging theorem}

Using the center manifold theorem, we have enventually reduced the original equation to a differential system of the form
\begin{equation}\left\{\begin{array}{ll} \dot{X}_h&=F\left(X_h,h_\ep\left(X_h,\theta\right),\theta\right)\\
\dot{\theta}&=\frac{1}{\varepsilon}
\label{SystRed}
\end{array}\right.
\end{equation}
and $Z_h(t)=h_\ep\left(X_h(t),\theta(t)\right)$. The function in the right-hand side of (\ref{SystRed}) can be expanded into powers of $\eps$ as follows
\begin{eqnarray*}
F(X,h_\eps(X,\theta),\theta) &=& F(X,0,\theta) + \eps \partial_Z F(X,0,\theta) h^1(X,\theta) + \eps^2 \partial_Z F(X,0,\theta) h^2(X,\theta)  \\
&+& \frac{\eps^2}{2} \partial_Z^2 F(X,0,\theta) (h^1(X,\theta),h^1(X,\theta)) + \ldots \\
&=& F_0(X,\theta) + \eps F_1(X,\theta) + \ldots + \eps^{r-1} F_{r-1}(X,\theta) + {\cal O}(\eps^{r}).
\end{eqnarray*}
The equation for $X$ being highly-oscillatory, it can be averaged according to the 
 following theorem.

\begin{theorem} \label{RefMoy2} For all $T_f >0$, there exists $\eps_0 >0$ such that for all $\eps <\eps_0$, there exists a change of variables $\tilde{\Phi}_\theta^{\eps}={\rm Id} +{\cal O}(\eps)$ and a function $\tilde{F}^\varepsilon$ defined on $\R^n$ satisfying the relation
\begin{eqnarray}
\forall t\in\left[0,T_f\right],\quad \| X \left(t\right)-\tilde{\Phi}_{\frac{t}{\eps}}^{\eps}\circ \tilde{\Psi}^\eps_{t}(x_0)\| \leq C \eps^{r}
\end{eqnarray}
where $\tilde{\Psi}^\eps_{t}$ is the flow of the differential equation with {\em autonomous} vector field $\tilde F^\eps$. 
\end{theorem}
\begin{remark} The first terms of $\tilde F^\eps = \tilde F_0 + \eps \tilde F_1 +\ldots$ are given by the formulas \cite{Chartier}
$$
\tilde F_0(X) = \frac{1}{T} \int_0^T F_0(X,\theta) d\theta = \frac{1}{T} \int_0^T F(X,0,\theta) d\theta
$$
and 
$$
\tilde F_1(X) =  \frac{1}{T} \int_0^T F_1(X,\theta) d\theta - \frac{1}{2T} \int_0^T \int_0^\theta [F_0(X,s),F_0(X,\theta)] ds d\theta
$$
where the Lie-bracket  stands for 
$$
[F_0(X,s),F_0(X,\theta)]:= \partial_X F_0(X,s) F_0(X,\theta) - \partial_X F_0(X,\theta) F_0(X,s).
$$
\end{remark}

Hence, the approximated differential system is:
\begin{itemize}
\item Up to order $0$ in $\ep$:
\begin{equation*}\left\{\begin{array}{ll} \dot{X}_h&=\tilde F_0(X) \\
\dot{\theta}&=\frac{1}{\varepsilon}
\end{array}\right.
\end{equation*}
\item Up to order $1$ in $\ep$:
\begin{equation*}\left\{\begin{array}{ll} \dot{X}_h&=\tilde F_0(X) +\ep \tilde F_1(X) \\
\dot{\theta}&=\frac{1}{\varepsilon}
\end{array}\right.
\end{equation*}
\end{itemize}

\subsection{Application to our system}
 
Now, we want to solve the equation which is verified by our new variable $X$ on the center manifold. The equation is the following:
\begin{equation}\left\{\begin{array}{ll} \dot{X}&=\begin{pmatrix}f^x_2\left(X,h(X,\theta),\theta\right)\\g^x_2\left(X,h(X,\theta),\theta\right)\end{pmatrix}\\
\dot{\theta}&=\frac{1}{\varepsilon}
\label{SysVar2}
\end{array}\right.
\end{equation}

%

To approximate the solution up to order 1 in $\varepsilon$, we make an averaging in $\theta$ for the functions. Theorem $\ref{RefMoy2}$ allows us to assert that this solution and the solution of $(\ref{SysVar2})$ are close to within $\mathcal{O}(\varepsilon)$.\\

We then denote:
\begin{align*}
\tilde F_0(X) &=\frac{1}{T}\int_0^T \begin{pmatrix}f^x_2\left(X,0,\theta\right)\\g^x_2\left(X,0,\theta\right)\end{pmatrix}\mathrm{d\theta} \\
\tilde F_1(X) &=\frac{1}{T}\int_0^T \begin{pmatrix}\partial_z f^x_2\left(X,0,\theta\right)h_p^1(X,\theta)\\ \partial_z g^x_2\left(X,0,\theta\right)h_q^1(X,\theta)\end{pmatrix}\mathrm{d\theta}\\ &\qquad\qquad-\frac{1}{2T}\int_0^T\int_0^\theta \left[\begin{pmatrix}f^x_2\left(X,0,s\right)\\g^x_2\left(X,0,s\right)\end{pmatrix}  ,\begin{pmatrix}f^x_2\left(X,0,\theta\right)\\g^x_2\left(X,0,\theta\right)\end{pmatrix}   \right]\mathrm{ds}\mathrm{d\theta}  
\end{align*}

Finally, we obtain the approximate system:
\begin{equation}\left\{\begin{array}{ll} \dot{X}&=\tilde F_0(X)+\ep\tilde F_1(X)\\
Z(t)&=\ep h^1\left(X(t),\theta(t)\right)
\label{SysVarMoy}
\end{array}\right.
\end{equation}

%

\begin{proposition} $ $ \\
 To express clearly the dependence in $x_p$ and $x_q$ of our differential system, we introduce the following notations:
\[h^0(X,\theta)=\begin{pmatrix} h_p^0(\theta) x_p(t) \\ h_q^0(\theta) x_q(t)  \end{pmatrix},\quad h^1(X,\theta)=\begin{pmatrix} h_p^1(\theta) x_p(t)- h_p^2(\theta) x_p(t) x_q(t) \\ h_q^1(\theta) x_q(t)- h_q^2(\theta) x_p(t) x_q(t)  \end{pmatrix}\] 
and we remind the notation: $\theta= \frac{t}{\ep}$.

\begin{itemize}
\item Up to order $0$ in $\ep$, the differential system is the following: 
\begin{equation}\left\{\begin{array}{ll}
 \dot{x_p}&=x_p(t) \left[\sum a_{p}\alpha_p^0 \right]- x_p(t) x_q(t) \left[\sum b_{p} \alpha_p^1\right]\vspace{2mm}\\
\dot{x_q}&=-x_q(t) \left[\sum a_{q}\alpha_q^0 \right]+ x_p(t) x_q(t) \left[\sum b_{q}\alpha_q^1\right]\vspace{2mm}\\
y_p(t)&=h_p^0(\theta) x_p(t)\\
y_q(t)&=h_q^0(\theta)x_q(t)
\end{array}\right.
\end{equation}
with 
\begin{align*}
\alpha_p^0 &= \frac{1}{T}\int_0^T\left(p_{eq}(\theta)+h_p^0(\theta) \right)\mathrm{d\theta}\\
\alpha_p^1 &= \frac{1}{T}\int_0^T\left(p_{eq}(\theta)+h_p^0(\theta) \right)\left(q_{eq}(\theta)+h_q^0(\theta) \right)\mathrm{d\theta}
\end{align*}
and the equivalent definitions for $\alpha_q^0$ and $\alpha_q^1$.\\

The solution $\tilde{x}_p$ and $\tilde{x}_q$ of this system are such that :
\begin{equation*}
\forall t\in\R,~\|\tilde{x}_p(t) - x_p(t)\|\leq C\ep, ~\|\tilde{x}_q(t) - x_q(t)\|\leq C\ep
\end{equation*}

\item Up to order $1$ in $\ep$, the differential system becomes: 
\begin{equation}\left\{\begin{array}{ll}
 \dot{x_p}&=x_p(t) \left[\sum a_{p} \alpha_p^0+\ep \sum a_{p} \beta_p^0
\right]\vspace{2mm}\\
&- x_p(t) x_q(t) \left[\sum b_{p} \alpha_p^1+\ep\sum b_{p} \beta_p^1+\ep\sum a_p \beta_p^2\right]\vspace{2mm}\\
& -\ep x_p (t) x_q(t)^2 \sum b_p \beta_p^3-\ep x_p (t)^2 x_q(t) \sum b_p \beta_p^4\vspace{2mm}\\
%
\dot{x_q}&=x_q(t) \left[\sum a_{q} \alpha_q^0+\ep \sum a_{q} \beta_q^0
\right]\vspace{2mm}\\
&- x_p(t) x_q(t) \left[\sum b_{q} \alpha_q^1+\ep\sum b_{q} \beta_q^1+\ep\sum a_q \beta_q^2\right]\vspace{2mm}\\
& -\ep x_p (t) x_q(t)^2 \sum b_q \beta_q^3-\ep x_p (t)^2 x_q(t) \sum b_q \beta_q^4\\
y_p(t)&=h_p^0(\theta) x_p(t)+\ep\left(h_p^1(\theta) x_p(t)-h_p^2(\theta) x_p(t)x_q(t)\right)\\
y_q(t)&=h_q^0(\theta) x_q(t)+\ep\left(h_q^1(\theta) x_q(t)-h_q^2(\theta) x_p(t)x_q(t)\right)
\end{array}\right.
\label{SystN2}
\end{equation}
with
\begin{align*}
\beta_p^0 &= \frac{1}{T}\int_0^T h_p^1(\theta) \mathrm{d\theta}-
\frac{1}{2T}\int_0^T\int_0^\theta \left[p_{eq}(s)+ h_p^0(s) ,p_{eq}(\theta)+ h_p^0(\theta) \right]\mathrm{ds}\mathrm{d\theta}   \\
\beta_p^1 &= \frac{1}{T}\int_0^T h_p^1(\theta) \left(q_{eq}(\theta)+h_q^0(\theta) \right)\mathrm{d\theta} + \frac{1}{T}\int_0^T h_q^1(\theta) \left(p_{eq}(\theta)+
h_p^0(\theta) \right)\mathrm{d\theta} \\&\qquad-\frac{1}{2T}\int_0^T\int_0^\theta \left[p_{eq}(s)+ h_p^0(s) ,p_{eq}(\theta)+ h_p^0(\theta) \right]\mathrm{ds}\mathrm{d\theta} \\
\beta_p^2&= \frac{1}{T}\int_0^T h_p^2(\theta) \mathrm{d\theta}\\
\beta_p^3&=\frac{1}{T}\int_0^T h_p^2(\theta) \left(q_{eq}(\theta)+h_q^0(\theta) \right) \mathrm{d\theta}\\
\beta_p^4&=\frac{1}{T}\int_0^T h_q^2(\theta) \left(p_{eq}(\theta)+h_p^0(\theta) \right) \mathrm{d\theta}\\
\end{align*}
and the corresponding notations for $q$.

The solution $\tilde{x}_p$ and $\tilde{x}_q$ of this system are such that :
\begin{align*}
\forall t\in\R,~\|\tilde{x}_p(t) - x_p(t)\|\leq C\ep, &~\|\tilde{x}_q(t) - x_q(t)\|\leq C\ep\\
\text{and for}~ t=k T, ~k\in\N, &~\|\tilde{x}_p(t) - x_p(t)\|\leq C\ep^2, ~\|\tilde{x}_q(t) - x_q(t)\|\leq C\ep^2
\end{align*}

\end{itemize}
\end{proposition}

To go to further order in $\ep$, we need to use Theorem $\ref{RefMoy2}$ and to evaluate the functions $\tilde{\Phi}^\varepsilon_t$ et $\tilde{F}^\varepsilon$ related to our problem. We solve:
\begin{equation}
\dot{\gamma}^\varepsilon=\tilde{F}^\varepsilon(\gamma^\varepsilon)
\end{equation}
If we denote $\tilde{\Psi}^\varepsilon_\theta(X_0)$ the flow related to this equation, we define then:
\[\tilde{X}(\theta)=\tilde{\Phi}^\varepsilon_\theta\circ\tilde{\Psi}^\varepsilon_\theta(X_0).\]

We can now assert that $X(\theta)$ and $\tilde{X}(\theta)$ are different up to $\mathcal{O}(e^{-\frac{C}{\varepsilon}})$. We get a correct approximation of our solution $X$. We still need to compute the related $Y$ vector, and then to perform the inverse changes of variables, to get an approximation up to $\mathcal{O}(e^{-\frac{C}{\varepsilon}}+e^{-\hat\mu t})$ of our first unknowns $p$ and $q$. If we use the approximation $h_{[n]}$ of the center manifold $h_\ep$, the solutions are close to within $\mathcal{O}(e^{-\frac{C}{\varepsilon}}+e^{-\hat\mu \frac{t}{\ep}}+\ep^{n+1})$.\\

We want to show that our method is an inprovement of the naive method, where we average in $\theta$ before we determine the center manifold. In other words, we call the naive method the one in which we average the equations before any study of the system. To do so, we study the stability of our system. We consider a situation in which $N=2$. According to the article \cite{Castella}, our differential system $(\ref{SystN2})$ has a stable (respectively unstable) equilibrium if:
\begin{align*}  \sigma&=\frac{\ep}{T}\left( \sum b_p\int_0^Th_q^2(\theta)\left(p_{eq}(\theta)+h_p^0(\theta) \right)\mathrm{d\theta}  \right. \\
&\qquad\qquad\left. +\sum b_q \int_0^Th_p^2(\theta)\left(q_{eq}(\theta)+h_q^0(\theta) \right)\mathrm{d\theta} \right) <0 ~(\text{resp.}~ >0).\end{align*}
We want to prove that the actual $\sigma$ can have a sign, and the $\sigma_0$ related to the naive method the other sign. Hence, the study of the stability gives different results with the naive method than with the real $\sigma$.\\

We choose:
\[ \tilde{K}_p=-1,~  \tilde{K}_q=-1,~p_{eq}(\theta)=\begin{pmatrix}1-a(\theta)\\a(\theta)\end{pmatrix},~q_{eq}=\begin{pmatrix}1-b\\b\end{pmatrix}\]
with $a(\theta)=a_0+a_1 \cos(\theta)+ a_{-1}\sin(\theta)$. The computation of $\sigma$ permits us to find suitable values of $a_0,a_1,a_{-1}$ and $b$.
Indeed, for $b_p=\begin{pmatrix}0.2\\0.1\end{pmatrix}$, $b_q=\begin{pmatrix}0.5\\0.3\end{pmatrix}$, $\ep=0.1$, if we choose $a_0=6$, $a_1=3$, $a_{-1}=2.5$ and $b=0.06$, we find a naive sigma $\sigma_0=-0.0122$, whereas the real $\sigma$ is equal to $0.0228$.

\section{An example with $N=2$}

We apply this method numerically on a simple example with two sites. Our equation is the following:
\begin{equation}\left\{\begin{array}{ll} 
\frac{\mathrm{d}p}{\mathrm{dt}}&=\frac{1}{\ep} K_p\left(\frac{t}{\ep}\right) + f(p,q)\\
\frac{\mathrm{d}q}{\mathrm{dt}}&=\frac{1}{\ep} K_q\left(\frac{t}{\ep}\right) + g(p,q)
\label{EqEx2d}
\end{array}\right.
\end{equation}
with the definitions:
\begin{align*}
K_p(t)&=\begin{pmatrix}-(\cos(t)+2)&\sin(t)+2\\\cos(t)+2&-(\sin(t)+2)\end{pmatrix}\\
K_q(t)&=\begin{pmatrix}-(\sin(t)+2)&\cos(t)+2\\\sin(t)+2&-(\cos(t)+2)\end{pmatrix}\\
f(p,q)&=\begin{pmatrix}a_p^1 p_1-b_p^1 p_1 q_1\\a_p^2 p_2-b_p^2 p_2 q_2 \end{pmatrix}\\
g(p,q)&= -\begin{pmatrix}a_q^1 q_1-b_q^1 p_1 q_1\\a_p^2 p_2-b_q^2 p_2 q_2 \end{pmatrix}
\end{align*}
For the numerical computation, we choose the following values:
\begin{equation*}
a_p=\begin{pmatrix}0.4\\0.3\end{pmatrix},~b_p=\begin{pmatrix}0.2\\0.1\end{pmatrix},~a_q=\begin{pmatrix}0.1\\0.2\end{pmatrix},~b_q=\begin{pmatrix}0.5\\0.3\end{pmatrix}
\end{equation*}
The direct approximation of the solution, using an implicit Euler method, with $dt=\ep^3$, gives us the Figure $\ref{SolExact}$, for $p_0=\begin{pmatrix} 0.1\\0.2\end{pmatrix}$ and $q_0=\begin{pmatrix} 0.3\\0.4\end{pmatrix}$.

\begin{figure}
\includegraphics[width=7cm]{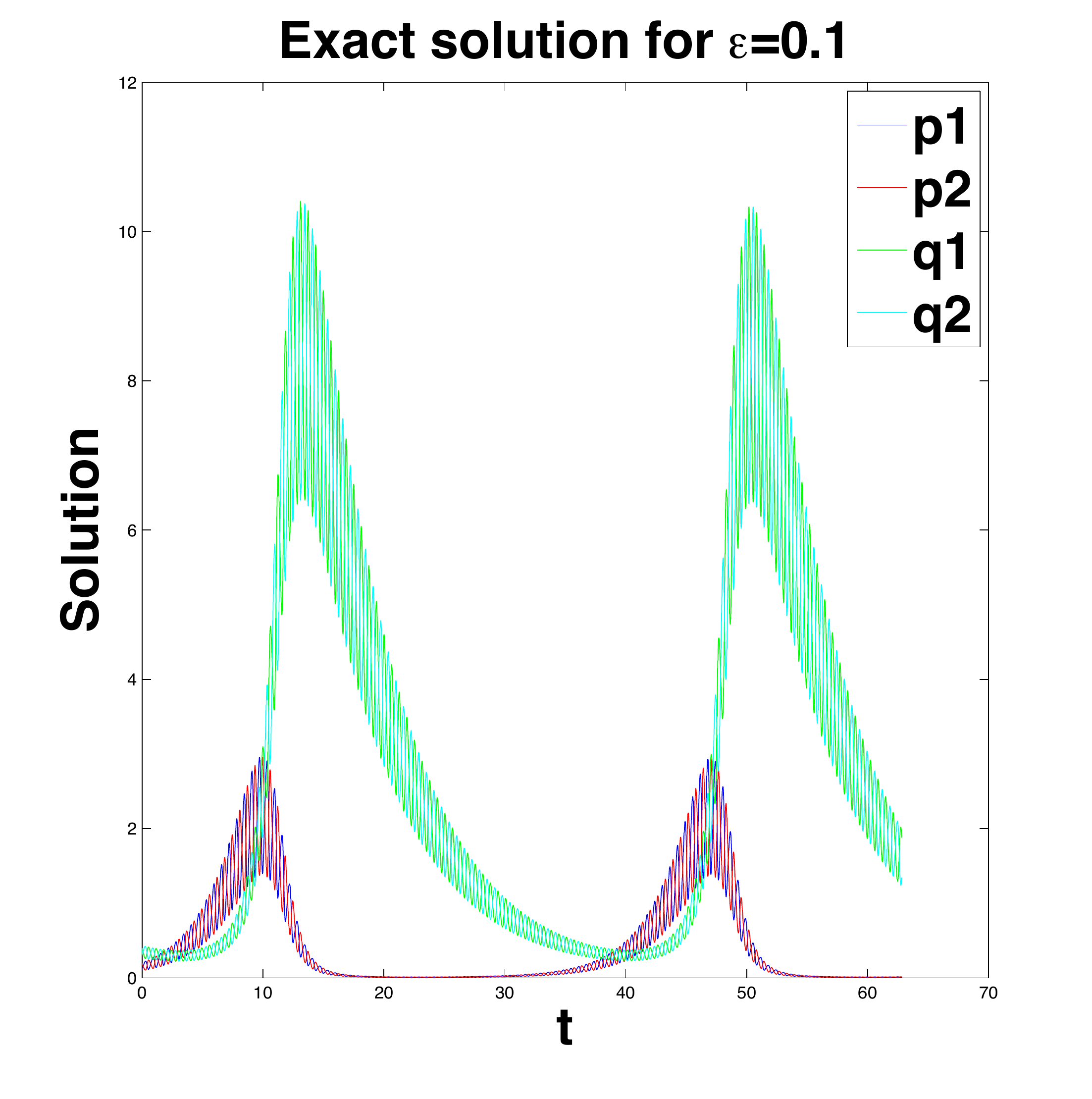}\hfill
\includegraphics[width=7cm]{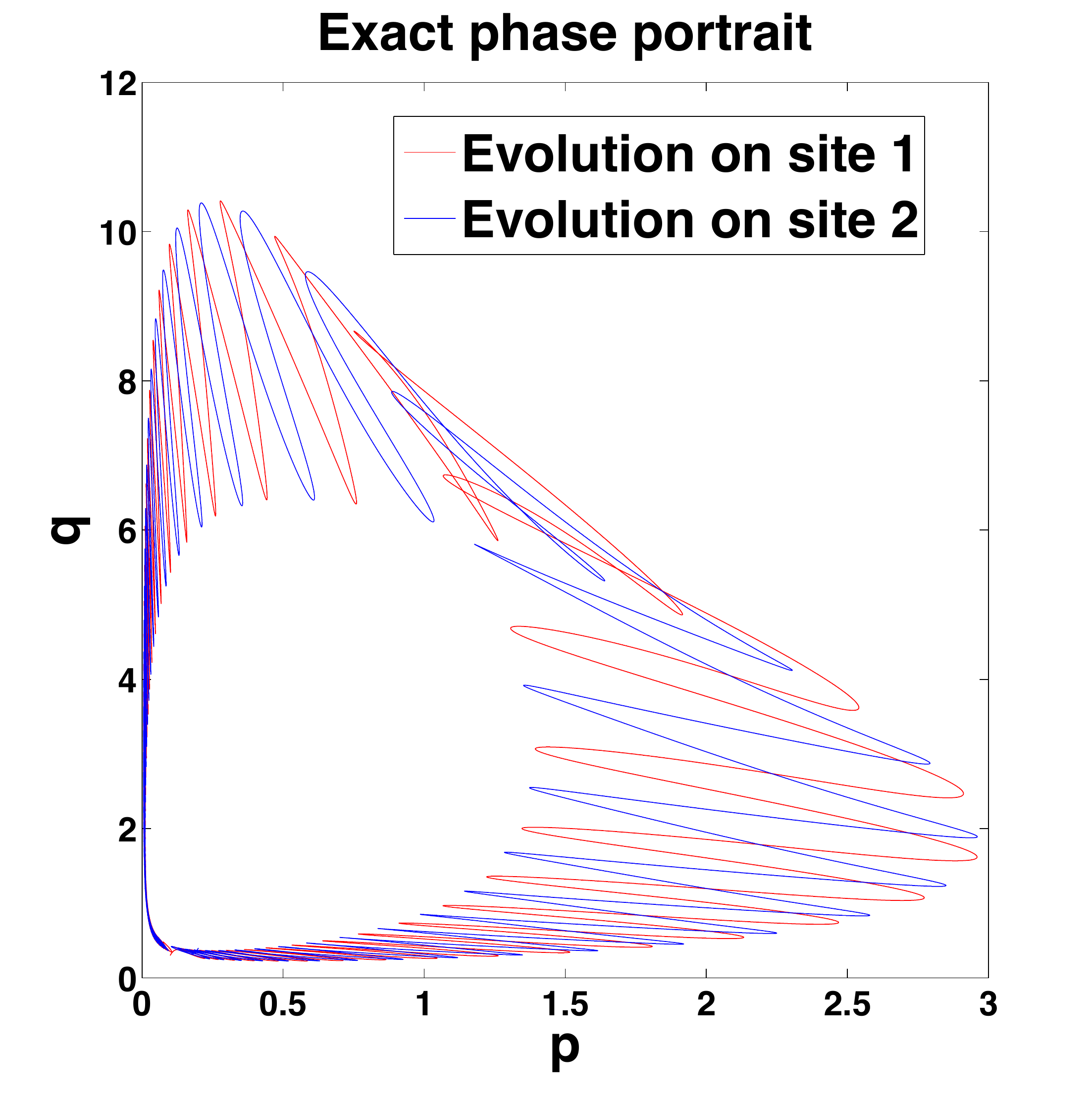}
\caption{Direct approximation of the solution $p_1$, $p_2$, $q_1$ and $q_2$ for $\ep=0,1$ and the corresponding phase portrait.}
\label{SolExact}
\end{figure}
\newpage
Now, we use our method to perform an approximation of this solution. We perform the change of variables, and to do so we compute $h_0$. We use the expressions:
\begin{align*}
h_p^0(x,\theta) &=- x \,\tilde R_p(\theta,0)(\mathrm{Id}-\tilde R_p(T,0))^{-1} \int_{\theta-T}^T \tilde R_p(0,\varphi) \dot{\tilde p}_{eq}(\varphi)d\mathrm{\varphi}\\
h_q^0(x,\theta) &=- x \,\tilde R_q(\theta,0)(\mathrm{Id}-\tilde R_q(T,0))^{-1} \int_{\theta-T}^T \tilde R_q(0,\varphi) \dot{\tilde q}_{eq}(\varphi)d\mathrm{\varphi}
\end{align*}
with $T=2\pi$. We approximate the integrals (using a Simpson method of order 3), and we get the Figure $2$ for $h_p^0$ and $h_q^0$, for $x_p=1$ and $x_q=1$.
\begin{figure}
\label{h0}
\includegraphics[width=7cm]{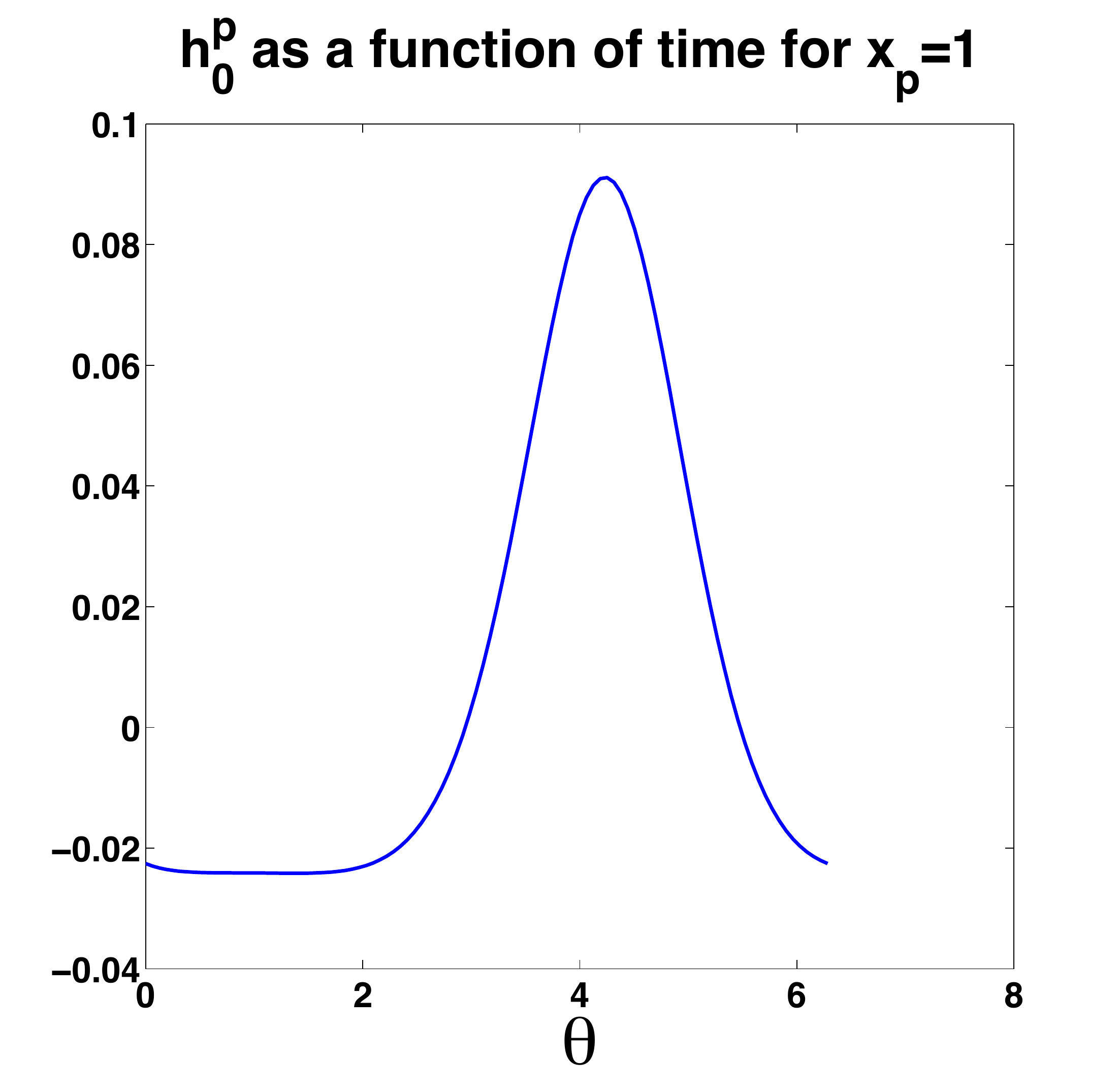}\hfill
\includegraphics[width=7cm]{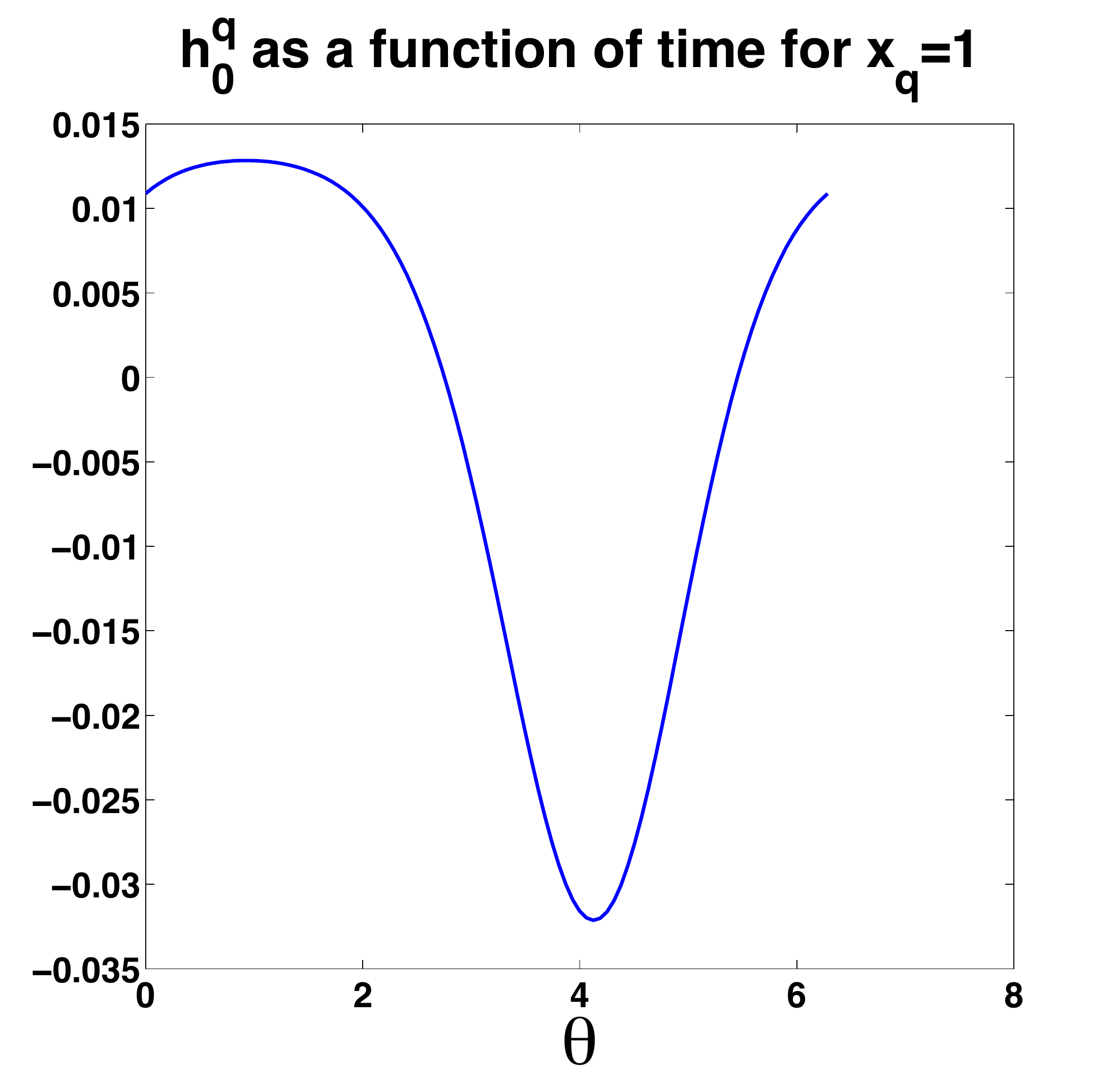}
\caption{First components of $h_p^0$ and $h_q^0$}
\end{figure}

Then, we approximate $h_1$, using the expression $(\ref{Ordre0})$:
\begin{equation*}
h^1(X,\theta) =R(\theta,0)(\mathrm{Id}-R(T,0))^{-1} \int_{\theta-T}^{\theta} R(0,\varphi) G(X,0,\varphi)  d\mathrm{\varphi}.
\end{equation*}
The Figure $3$ shows the evolution of $h_p^1$ and $h_q^1$, for $x_p=1$ and $x_q=1$.
\begin{figure}
\label{h1}
\includegraphics[width=7cm]{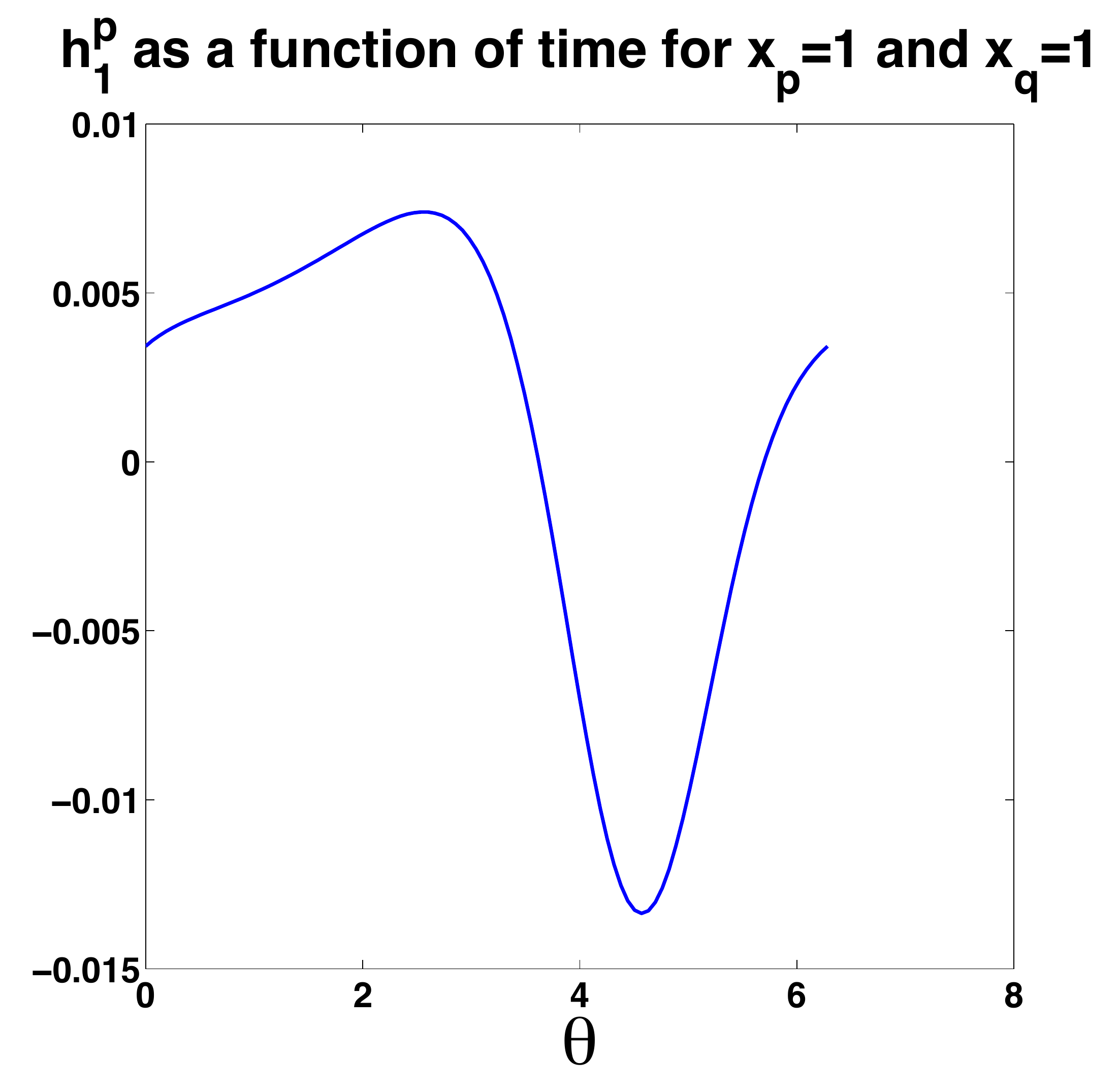}\hfill
\includegraphics[width=7cm]{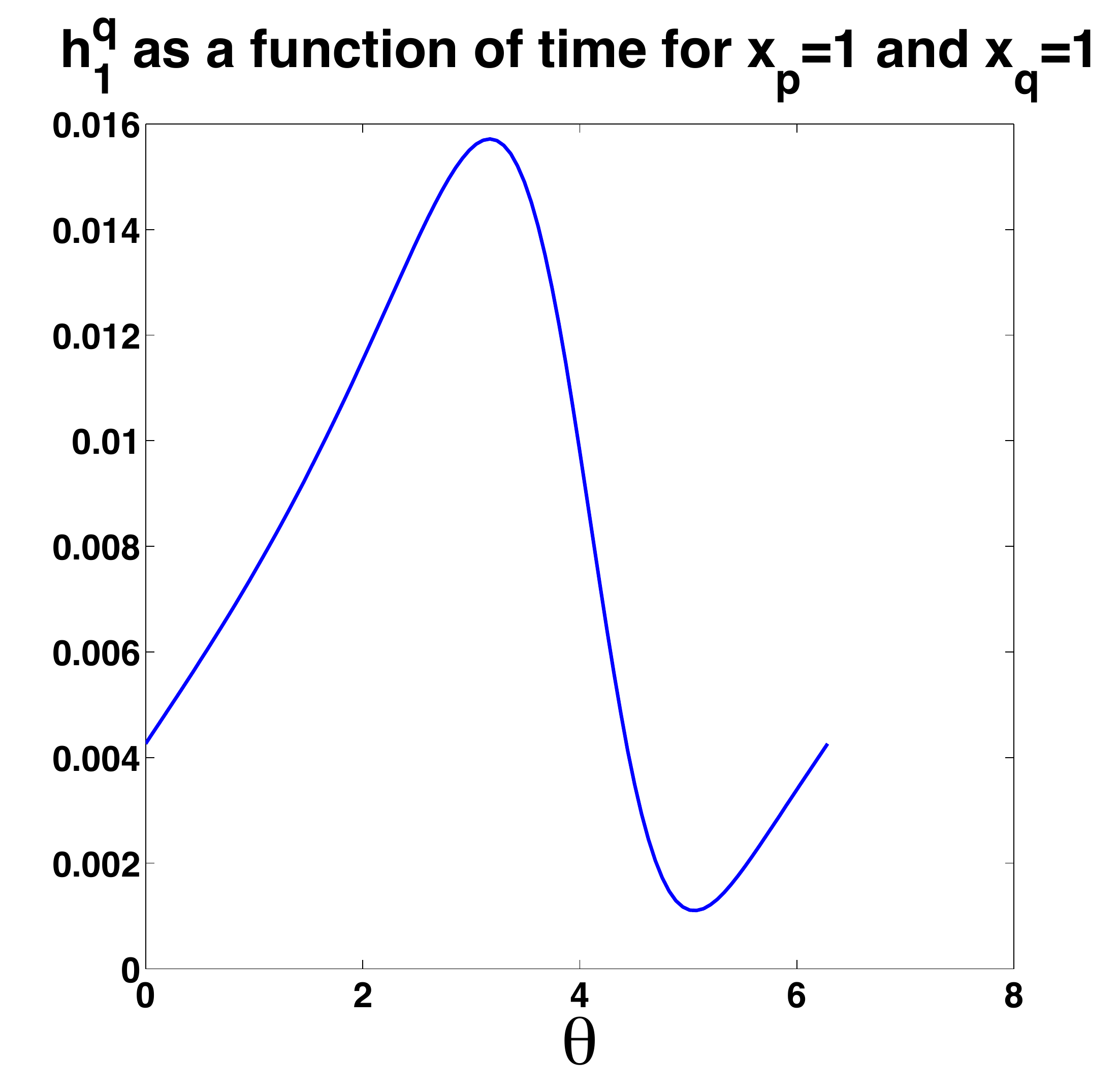}
\caption{First components of $h_p^1$ and $h_q^1$}
\end{figure}

\newpage
Then, we solve the equation in $X$, using a RK4 method. We do not use averaging here because we want to illustrate the impact of the approximation of the center manifold $h_\ep$ on the error. Performing the inverse change of variables, we get the following approximate solutions for $p_0=\begin{pmatrix} 0.1\\0.2\end{pmatrix}$ and $q_0=\begin{pmatrix} 0.3\\0.4\end{pmatrix}$, presented in Figure 4.

\begin{figure}[h]
\label{SolApprox}
\includegraphics[width=14cm]{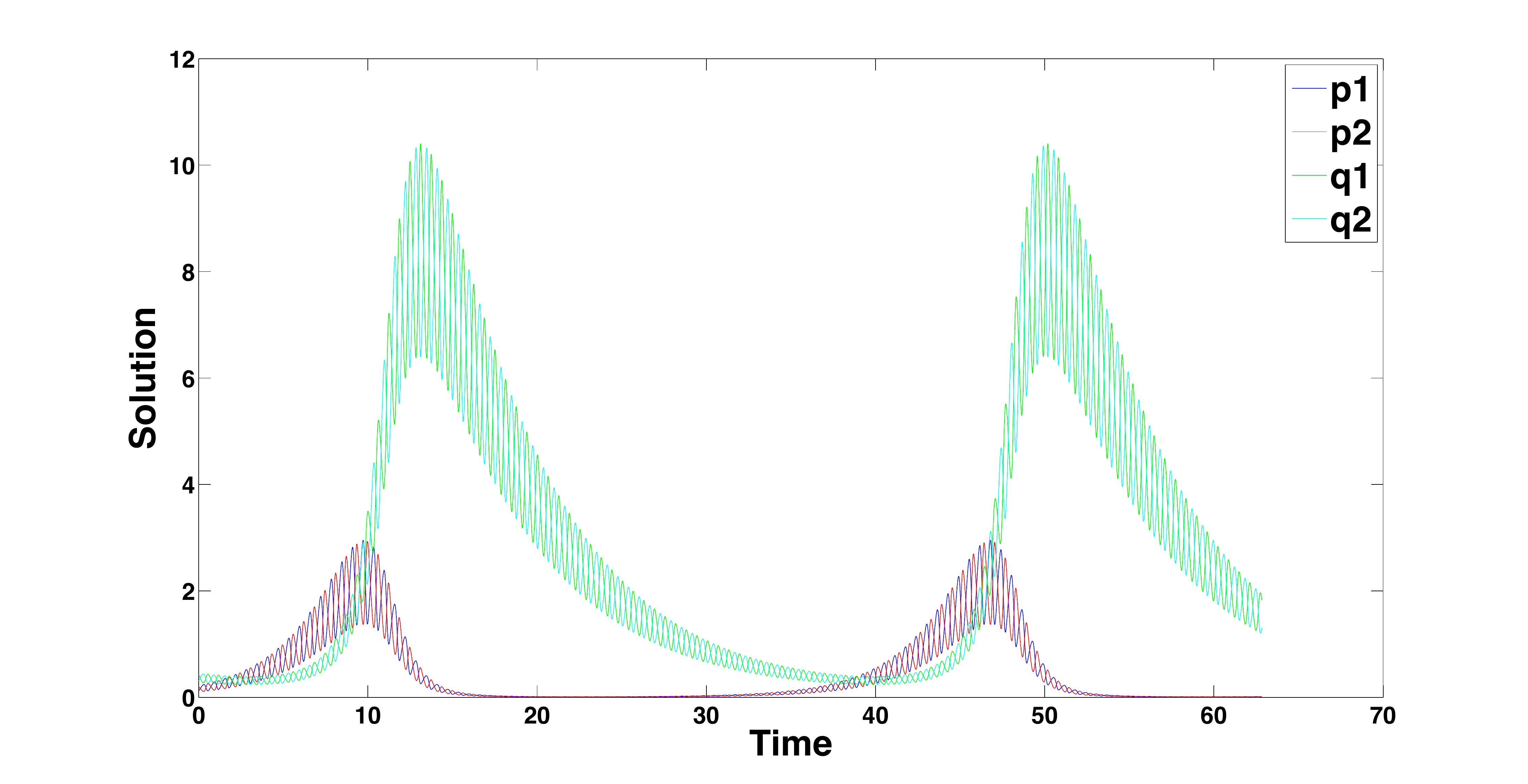}
\caption{Approximation of the solution for $\ep=0.1$ and $dt=0.01$}
\end{figure}
\newpage

We study the error. When we do not choose initial conditions $(p_0,q_0)$ corresponding to a $(x_0,z_0)$ on the central manifold, we observe an exponential decrease in $\frac{t}{\ep}$ of the error towards an $\mathcal{O}\left(\ep\right)$. Figure $\ref{fig01}$ shows the evolution of the error as a function of time for $\ep=0.1$, and Figure $\ref{fig001}$ for $\ep=0.01$.

\begin{figure}
\includegraphics[width=7cm]{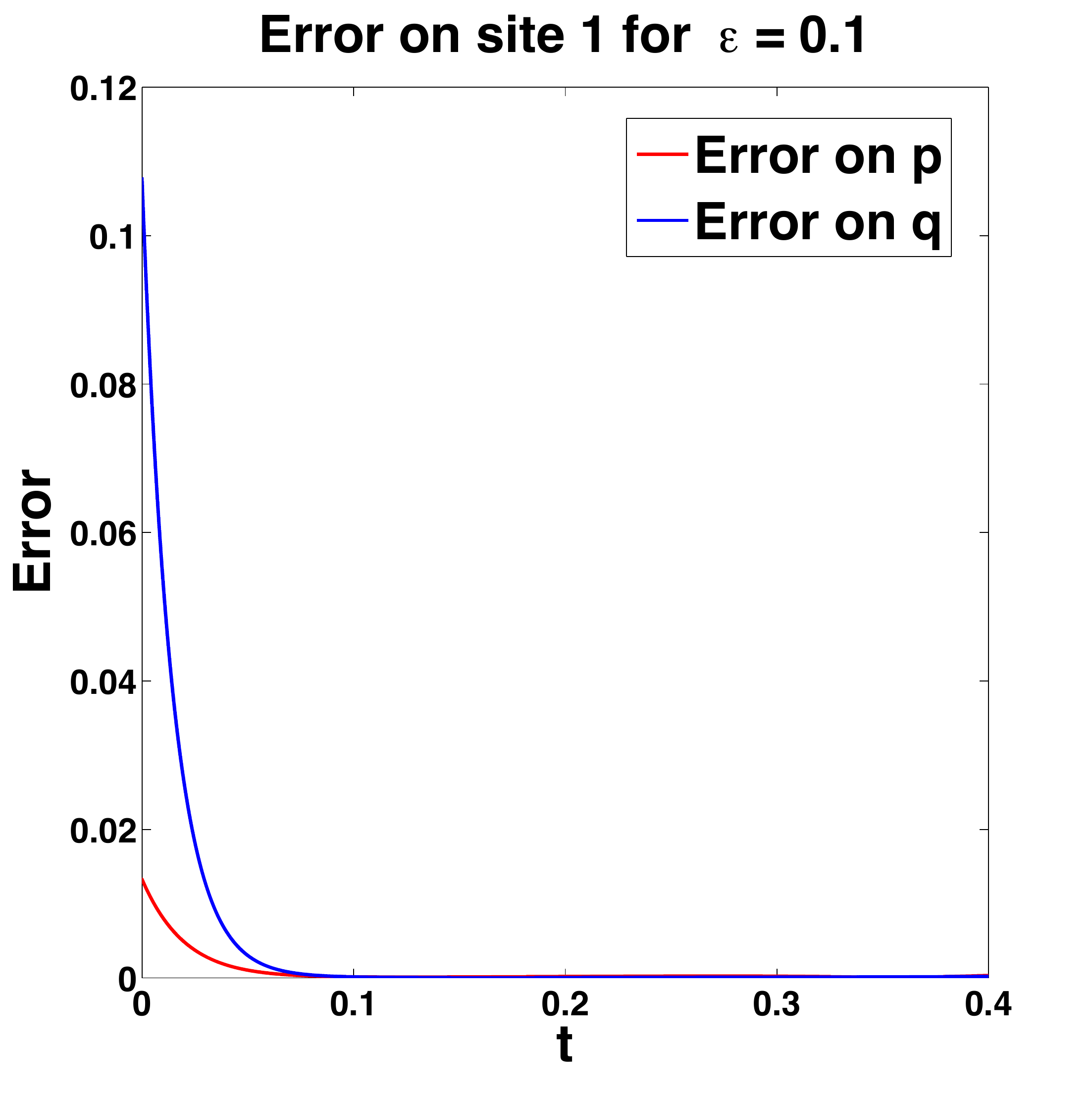}\hfill
\includegraphics[width=7cm]{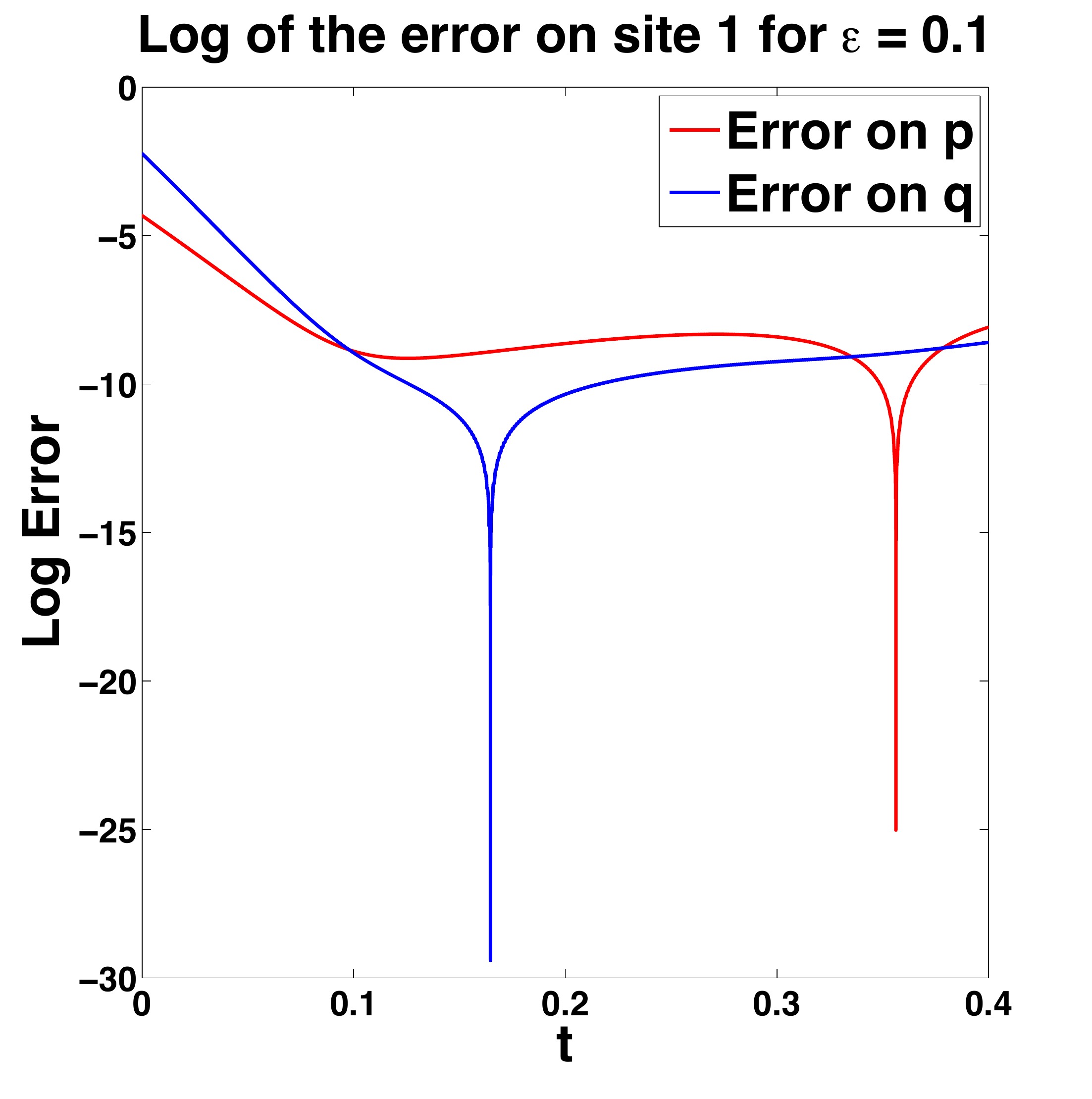}
\caption{Error between exact and approximate solution on site 1 for $\ep=0.1$, and log of the error.}
\label{fig01}
\end{figure}
\begin{figure}
\includegraphics[width=7cm]{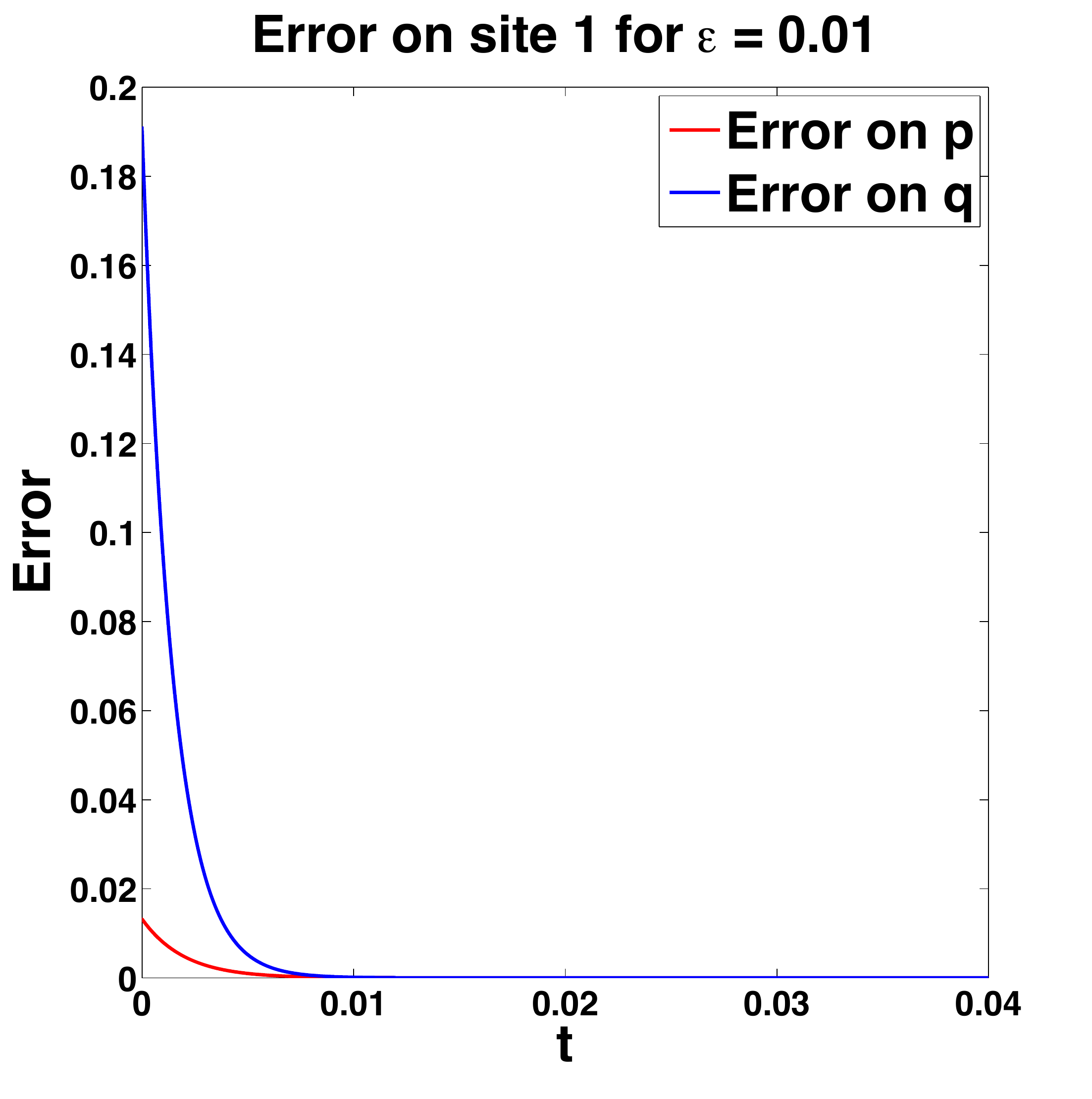}\hfill
\includegraphics[width=7cm]{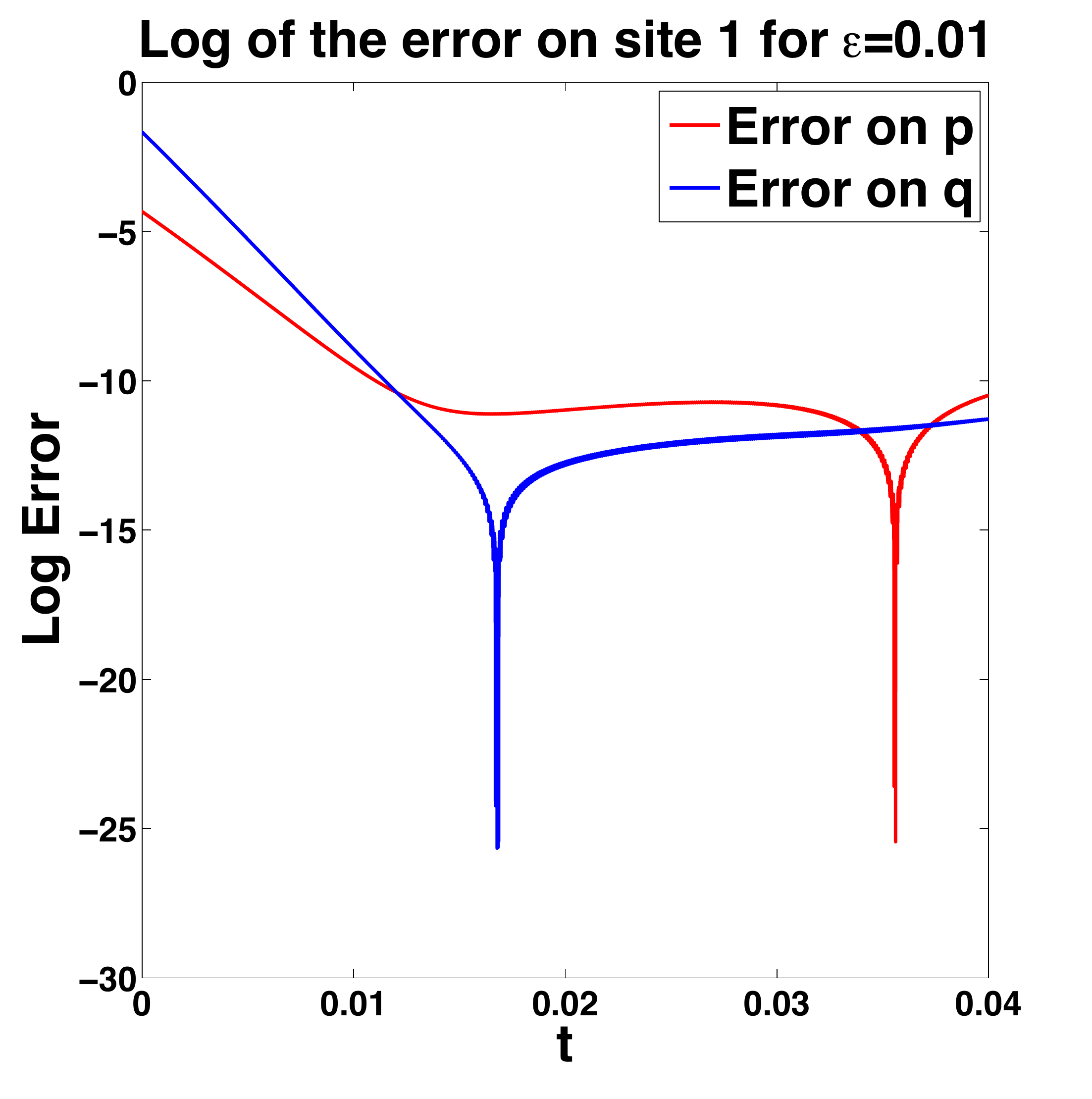}
\caption{Error between exact and approximate solution on site 1 for $\ep=0.01$, and log of the error.}
\label{fig001}
\end{figure}

To analyse the impact of the approximation of the central manifold, we change the inital conditions to begin on the central manifold. Then, we compute the maximum of the error between the real solution and the approximate one on the intervall $[0;10\ep]$ for differents value of $\ep$. Using a logaritheoremic scale, we obtain the Figure 7. Hence, we have a slope $1.0386$ for the error on $p$, and $1.0080$ for the error on $q$, and so we have shown that the first order approximation of the central manifold leads to an error in $\mathcal{O}\left(\ep\right)$.

\begin{figure}[h]
\label{LErrLep}
\includegraphics[width=12cm]{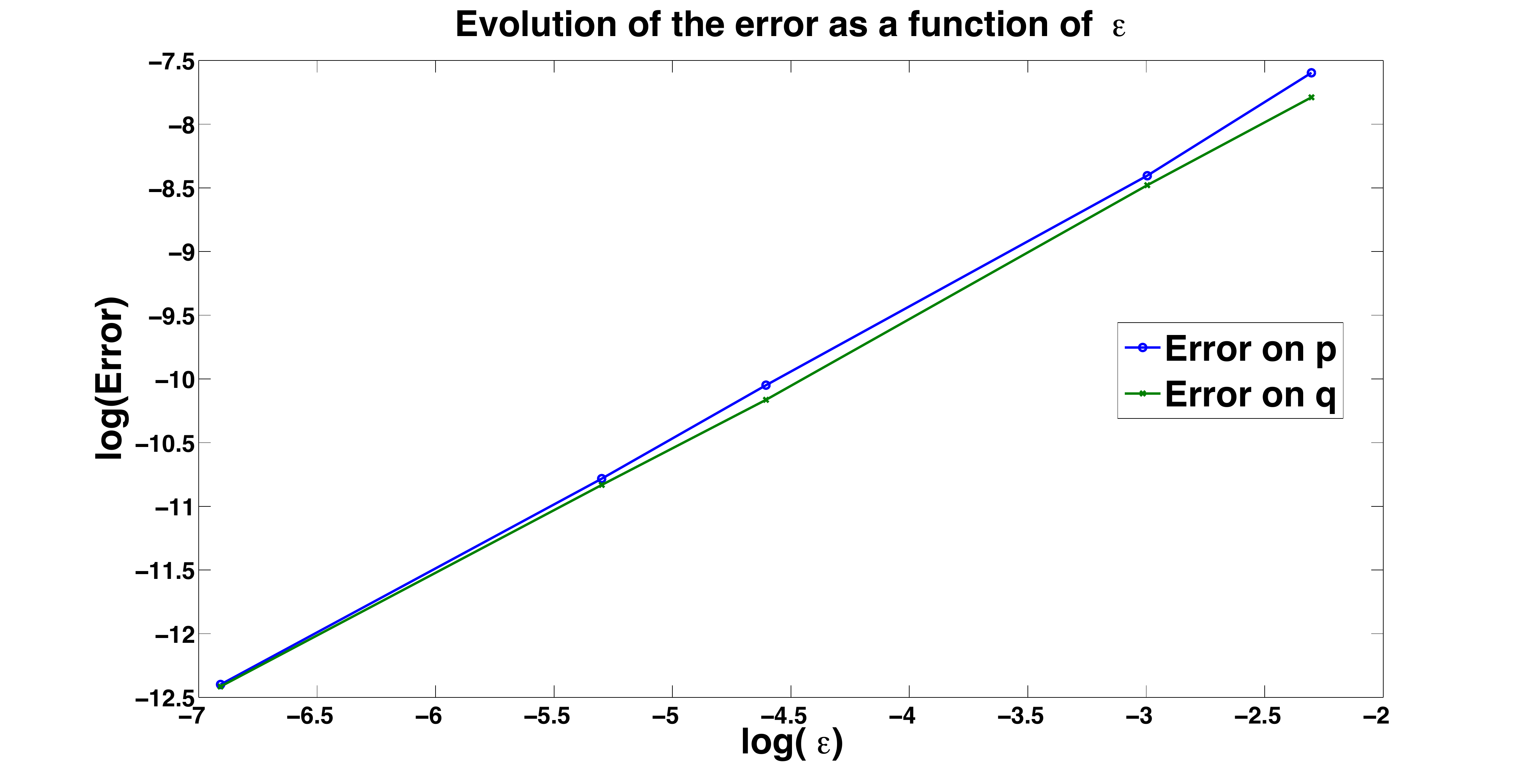} 
\caption{Evolution of the error as a function of $\ep$ (log-log scale)}
\end{figure}

The last part of the numerical resolution can be computed using averaging techniques, as presented in \cite{Chartier}.

 


\newpage
\addcontentsline{toc}{section}{Bibliography}
\nocite{*}

\bibliographystyle{alpha}

\bibliography{BibliArticle}

%
%

\end{document}